\documentclass[11pt]{article}

\usepackage[round]{natbib}

\usepackage{graphbox}

\usepackage{color}
\usepackage{float}
\restylefloat{table}

\usepackage[title,titletoc,toc]{appendix}

\usepackage{mathtools}
\usepackage{bigints}
\usepackage{amsthm}
\usepackage{graphicx}
\usepackage{pdfsync}
\usepackage{dsfont}
\usepackage{mathrsfs}
\usepackage{amssymb}
\usepackage{mathptmx}

\usepackage{pgf,tikz}
\usepackage{pgfplots}
\usetikzlibrary{automata}
\usetikzlibrary{arrows}
\usetikzlibrary{matrix,positioning}
\usetikzlibrary{snakes}
\usetikzlibrary{decorations.pathreplacing}

\usetikzlibrary{decorations.markings}
\usetikzlibrary{decorations.shapes}
\tikzset{decoration={brace}}

\numberwithin{equation}{section}
\usepackage{bm}
\newtheorem{theorem}{Theorem}[section]
\newtheorem{lemma}[theorem]{Lemma}

\theoremstyle{definition}
\newtheorem{example}[theorem]{Example}
\newtheorem{corollary}[theorem]{Corollary}
\newtheorem{remark}[theorem]{Remark}
\theoremstyle{remark}
\usepackage{tensor}
\usepackage{comment}

\newcommand{\Exp}{\mathds{E}}

\newcommand{\Prob}{\mathds{P}}

\newcommand{\dd}{\mathrm{d}}
\newcommand{\e}{\mathrm{e}}

\newcommand{\vect}[1]{\vec{#1}}
\renewcommand{\vect}[1]{\boldsymbol{\bm #1}}
\newcommand{\mat}[1]{\boldsymbol{\bm #1}}

\newcommand{\disabfig}
{
\begin{figure}[htb]
   \centering  
\begin{tikzpicture}[scale=0.75] 

\draw[thin] (0,7)  node {$\bullet$} node[above] {$s$} -- (10,7)  node {$\bullet$} node[above] {$T$}; 
\draw[thick,green] (0,7) -- (10,7); \draw (11,7)  node[right] {$F_0=\{\tau_0^1\wedge\tau_0^2>T\}$}; 
\draw (7,7) node {$\bullet$} node[above] {$S$};

\draw[thin] (0,6)  node {$\bullet$} -- (10,6)  node {$\bullet$}; \draw (7,6) node {$\bullet$};
\draw[thick,green] (0,6) -- (5,6); \draw (5,6) node{$*$} node[below] {$\tau_0^2$}; 
\draw[thick,red] (5,6) -- (10,6); 
\draw (11,6)  node[right] {$F_{2-}=\{0<\tau_0^2\le S\}$}; 

\draw[thin] (0,5)  node {$\bullet$} -- (10,5)  node {$\bullet$}; \draw (7,5) node {$\bullet$};
\draw[thick,green] (0,5) -- (8,5); \draw (8,5) node{$*$} node[below] {$\tau_0^2$}; 
\draw[thick,red] (8,5) -- (10,5); 
\draw (11,5)  node[right] {$F_{2+}=\{S<\tau_0^2\le T\}$}; 

\draw[thin] (0,4)  node {$\bullet$} -- (10,4)  node {$\bullet$}; \draw (7,4) node {$\bullet$};
\draw[thick,green] (0,4) -- (5,4); \draw (5,4) node{$*$} node[below] {$\tau_0^1$}; 
\draw[thick,blue] (5,4) -- (10,4); 
\draw (11,4)  node[right] {$F_{1-}=\{0<\tau_0^1\le S,\, \tau_1^2>T\}$}; 

\draw[thin] (0,3)  node {$\bullet$} -- (10,3)  node {$\bullet$}; \draw (7,3) node {$\bullet$};
\draw[thick,green] (0,3) -- (8,3); \draw (8,3) node{$*$} node[below] {$\tau_0^1$}; 
\draw[thick,blue] (8,3) -- (10,3); 
\draw (11,3)  node[right] {$F_{1+}=\{S<\tau_0^1 ,\, \tau_1^2>T\}$}; 

\draw[thin] (0,2)  node {$\bullet$} -- (10,2)  node {$\bullet$}; \draw (7,2) node {$\bullet$};
\draw[thick,green] (0,2) -- (5,2); \draw (5,2) node{$*$} node[below] {$\tau_0^1$}; 
\draw[thick,blue] (5,2) -- (6.5,2); \draw (6.5,2) node{$*$} node[below] {$\tau_1^2$};  \draw[thick,red] (6.5,2) -- (10,2); 
\draw (11,2)  node[right] {$F_{1_-2_-}=\{0<\tau_0^1<\tau_1^2\le S\}$}; 

\draw[thin] (0,1)  node {$\bullet$} -- (10,1)  node {$\bullet$}; \draw (7,1) node {$\bullet$};
\draw[thick,green] (0,1) -- (5,1); \draw (5,1) node{$*$} node[below] {$\tau_0^1$}; 
\draw[thick,blue] (5,1) -- (9,1); \draw (9,1) node{$*$} node[below] {$\tau_1^2$};  \draw[thick,red] (9,1) -- (10,1); 
\draw (11,1)  node[right] {$F_{1_-2_+}=\{0<\tau_0^1\le S<\tau_1^2\le T\}$}; 

\draw[thin] (0,0)  node {$\bullet$} -- (10,0)  node {$\bullet$}; \draw (7,0) node {$\bullet$};
\draw[thick,green] (0,0) -- (7.5,0); \draw (7.55,0) node{$*$} node[below] {$\tau_0^1$}; 
\draw[thick,blue] (7.5,0) -- (9,0); \draw (9,0) node{$*$} node[below] {$\tau_1^2$};  \draw[thick,red] (9,0) -- (10,0); 
\draw (11,0)  node[right] {$F_{1_+2_+}=\{0<\tau_0^1\le S<\tau_1^2\le T\}$}; 

\end{tikzpicture}
\caption{The sample space partition. State 0=active=green, State 1=disabled=blue, State 2=dead=red} 
\label{disabfig}
\end{figure}
}

 \newcommand{\varphiFig}{
  \begin{figure}[htb]
   \centering  
\begin{tikzpicture}[xscale=1.5,yscale=0.5] 
 \draw[thick] (0,5) to [out=-85, in=180] (2,1) -- (3.8,1);
 \draw[thin,<->] ((4,0) node[below] {$t$}-- (0,0) -- (0,6);
 \draw[thin,dashed] (2,0) node[below] {$a$} -- (2,1) (0,2) node[left] {$y=\varphi(t)$} --(0.55,2);
  \draw[thin,dashed] (0.55,0) node[below] {$t=\psi(y)$} -- (0.55,2) (0,1) node[left] {$b$} -- (2,1);
  \draw (0,5) node[left] {$b+a^2$};
\end{tikzpicture}
\caption{\mbox{$\varphi(t)$}}
\label{varphiFig}
\end{figure}
}

\begin{document}
\title{Matrix calculations for inhomogeneous Markov reward processes, with applications to life insurance and point processes}
%\date{Draft,\ \today}

\author{Mogens Bladt$^1$, S\o ren Asmussen$^2$,  and Mogens Steffensen$^{3}$,\\
\small
1. University of Copenhagen, Department of Mathematical Sciences;
bladt@math.ku.dk \\
\small
2. Aarhus University, Department of Mathematics; asmus@math.au.dk 
\\
\small 
3. University of Copenhagen, Department of Mathematical Sciences;
mogens@math.ku.dk }
\maketitle

\begin{abstract}
A multi--state life insurance model is naturally described in terms of the intensity matrix
of an underlying (time--inhomogeneous) Markov process which describes
the dynamics for the states of an insured person. Between and at transitions, benefits and premiums
are paid, defining a payment process, and the technical reserve is defined as the present value
of all future payments of the contract. Classical methods for finding the reserve and higher order moments 
involve the solution of certain differential equations (Thiele and Hattendorf, respectively). In this paper
we present an alternative matrix--oriented approach based on general reward considerations for Markov jump
processes. The matrix approach provides a general framework for effortlessly setting up general and even 
complex multi--state models, where moments of all orders are then expressed explicitly in terms of so--called product integrals (matrix--exponentials) of certain matrices. 
As Thiele and Hattendorf type of theorems can be retrieved immediately from the matrix formulae, this methods also provides a quick and transparent approach to proving these classical results. Methods for obtaining distributions and related properties of interest (e.g. quantiles or survival functions) of the future payments are presented from both a theoretical and practical point of view (via Laplace transforms and methods involving orthogonal polynomials).
\end{abstract}

\section{Introduction}
In this paper we consider the distribution and moments of the total reward generated by an time inhomogeneous Markov process with a finite state space. Rewards may be earned in three different ways. During sojourns in a fixed state, rewards can be earned at a deterministic (time--dependent) rate and as deterministic (time--dependent) lump sums which arrive according to a non--homogeneous Poisson process. Finally, at the times of transition deterministic (time-dependent) lump sums, which also depends on the type of transition, may be earned w.p.\ 1 or at random with some probability.  We are particularly interested in the case of discounted rewards which have applications in life insurance. Here the rewards (premiums and benefits) are discounted by a deterministic (though time--dependent) interest rate. This setting is slightly more general than the standard life insurance set--up, and is inspired by the parametrisation of the general Markovian Arrival Process (MAP). In this way we also achieve the calculation of rewards and moments in the MAP and its time--inhomogeneous extension. 

Our method for deriving distributions and moments uses probabilistic (sample path) arguments and matrix algebra. In particular, the matrices of interest are readily derived from the intensity matrix of the underlying Markov process and a matrix of payments. This is true for both pure and discounted rewards, where in the latter case the interest rate may be accommodated conveniently into the intensity matrix. The Laplace transform for the total (discounted) reward is obtained as a product integral (which is a matrix exponential in the time--homogeneous case) involving these matrices. Concerning the moments, all moments up to order $k$ are obtained by a product integral of a $(k+1)\times (k+1)$ block matrix build upon the aforementioned matrices. The product integrals may be evaluated in a number of ways. Generally, the product integral satisfies a system of linear differential equation of the Kolmogorov type from which we retrieve both Thiele's differential equation and Hattendorff type of theorems. If the intensities and payments are piecewise constant, which may often be the case in practical implementations, the product integral reduces to a product of ma{}trix--exponentials which may be evaluated numerically by efficient methods like e.g. uniformisation. 

While higher order moments are rarely used in life insurance, our approach to general accessibility of all orders and their numerical computability up to quite high orders suggest that they can also be used for approximating the cumulative distribution function (c.d.f.) of the distribution, and thereby the calculation of quantiles (values at risk or confidence intervals) which could provide valuable information concerning the actual risk. We provide a first example along this
line by proposing a Gram-Charlier expansion to approximating both the density (p.d.f.) and the c.d.f.\ of the discounted future payment distribution when this is absolutely continuous. 
While this requires at least one continuous payment stream (e.g.\ 
 a premium) to be present, we will see that shape of the distribution can be challenging, particularly for the case of the p.d.f. 

The idea of using  multi-state (time inhomogeneous) Markov processes as a model in life insurance dates back at least to the 1960's and was put into a modern context by \cite{Hoem1969}, in which also Thiele's differential equations for the state--wise reserves are derived. Variance formulas for the future payments can e.g.\ be found in \cite{Ramlau1988} whereas for higher order moments we refer to \cite{NorbergRagnar1994Defm}. A differential equation (Thiele) approach to calculating the c.d.f.\ of the discounted future payments has been considered in \cite{HESSELAGER1996}. If one considers only unit rewards on all jumps and Poisson arrivals (no continuous rewards) in a time--homogeneous Markov jump process, then the total undiscounted reward up to time $t$ defines a point process which is known as a Markovian Arrival Process (MAP). The moments in the MAP have been shown to satisfy certain integral equations (which are equivalent to the differential equations from life insurance) as shown in \cite{bo-Uffe-2007}. Apart from defining a tractable class of point process with numerous applications in applied probability, the MAPs form a simple dense class of point processes on on the positive reals (see \cite{Asmussen:1993tn}).

The specific contributions of the paper are as follows. The Laplace transform of the total rewards (Theorem \ref{Th:4.5a}) generalises a similar result of \cite{bla:02} for time--homogeneous Markov processes. The matrix representation of moments depends on a crucial result given in Lemma~\ref{lemma:van-loan}, which generalises a similar result proved in \cite{VanLoan:1978tq} for the case of constant matrices. The reserves and moments we deal with in the analysis are the so--called partial reserves and moments, which are defined as the expected value of the (powers) of the future (discounted) payments contingent on the terminal state. These reserves and moments may well be of interest on their own. The matrix representation of all moments provides a unifying approach to the explicit solution of Thiele's and Hattendorff type of theorems. Working with solutions rather than the corresponding differential equations may greatly simplify subsequent analysis and applications.

The rest of the paper is organised as follows. In Section \ref{sec:background} we review the basic properties of the product integral, which will play an important role in the paper. The basic model and notation is set up in Section \ref{sec:model} and in Section \ref{sec:Thiele} we use a probabilistic argument to prove a slightly extended version of Thiele's differential equation. An important technical result regarding the calculation of certain ordinary integrals via product integrals is proved in Section \ref{sec:matrix-reserves}. The main construction takes place in Sections \ref{sec:laplace-transform} and \ref{sec:moments} where we derive explicit matrix representations for the Laplace transform and higher order moments of the discounted future payments (total reward). An slightly extended version of Hattendorff's theorem is derived as a consequence in the end of Section \ref{sec:moments}. As an example, we also calculate the higher order moments and factorial moments in a (time--homogeneous) Markovian Arrival Process. 
Since the moments of up to high orders are easily calculated, in Section 9 we explore the possibility of calculating the p.d.f.\ and c.d.f.\ for the total discounted future payments by means of orthogonal polynomial expansions based on central moments. In Section 9 we provide a numerical example, and in  Section 10 we conclude the paper.

\section{Some relevant background}\label{sec:background}
Let $\mat{A}(x)=\{ a_{ij}(x)  \}_{i,j=1,...,p}$ be a $p\times p$ matrix function. 
The product integral of $\mat{A}(x)$, written as 
\[ \mat{F}(s,t) = \prod_s^t (\mat{I}+ \mat{A}(x)\,\dd x)   , \]
where $\mat{I}$ denotes the identity matrix, may be defined in a number of equivalent ways. 
It is e.g.\ the solution to Kolmogorov forward differential equation 
\[  \frac{\partial}{\partial s}\mat{F}(s,t)=-\mat{A}(s)\mat{F}(s,t), \ \ \ \mat{F}(t,t)=\mat{I}, \]
which by integration and repeated substitutions also yields the Peano--Baker series representation
\begin{equation}
  \prod_s^t (\mat{I}+\mat{A}(x)\,\dd x)=\mat{I} + \sum_{n=1}^\infty \int_s^t\int_s^{x_n}\cdots \int_s^{x_{2}}\mat{A}(x_1)\mat{A}(x_2)\cdots \mat{A}(x_n)\,\dd x_1\,\dd x_2\cdots \,\dd x_n , \label{eq:Peano-Baker} .
\end{equation} 
 The product integral exists if $\mat{A}(x)$ is Riemann integrable, which will henceforth be assumed. From the Peano--Baker representation one may prove that
\begin{eqnarray}
\prod_s^t (\mat{I}+\mat{A}(x)\,\dd x)&=&\prod_s^x (\mat{I}+\mat{A}(u)\,\dd u)\prod_x^t (\mat{I}+\mat{A}(u)\,\dd u)   \label{eq:product-rule}
\end{eqnarray}
holds true for any order of $s,t$ and $x$ (not only $s\leq x\leq t$). In particular, the inverse of a product integral then exists and is given by
 \begin{equation}
  \left( \prod_s^t (\mat{I}+\mat{A}(x)\,\dd x)\right)^{-1} =\prod_t^s (\mat{I}+\mat{A}(x)\,\dd x) \label{eq:inverse-prod-int}  .
\end{equation} 
If the matrices $\mat{A}(x)$ all commute, then
\begin{equation}
 \prod_s^t (\mat{I}+\mat{A}(x)\,\dd x) =\exp \left( \int_s^t \mat{A}(x)\,\dd x  \right) .\label{eq:prod-int-commute}
\end{equation}
In particular, if $\mat{A}(x)=\mat{A}$ for all $x$, then
\begin{equation}
 \prod_s^t (\mat{I}+\mat{A}(x)\,\dd x) =\exp \left( \mat{A}(t-s) \right) . \label{eq:prod-int-constant}
\end{equation}
This last observation may be useful in connection with piecewise constant matrices $\mat{A}(x)$
\[  \mat{A}(x) = \mat{A}_i,\ \ \ x_{i-1}\leq x \leq x_i   \]
for $i=1,2,...$ and where $x_0=0$. Then, using \eqref{eq:product-rule}, 
 \begin{equation}
 \prod_s^t (\mat{I}+\mat{A}(x)\,\dd x) = \exp \left( \mat{A}_i(x_i-s) \right) \left( 
 \prod_{k=i+1}^{j-1} \exp \left( \mat{A}_k(x_k-x_{k-1}) \right) \right) \exp \left( \mat{A}_{j}(t-x_{j-1}) \right) 
\label{eq:piecewise-constant}
 \end{equation}
 , 
 where $i$ and $j$ are such that $s\in [x_{i-1},x_{i}]$ and $t\in [x_{j-1},x_{j}]$. 

 Matrix--exponentials can be calculated in numerous ways (see \cite{Moler:1978vp} and \cite{Loan:2003un}) and are typically available in standard software package though at varying level of sophistication. If the exponent is an intensity (or sub--intensity matrix, i.e. row sums are non--positive) then either a Runge--Kutta method or uniformisation (see e.g.\  \cite{bladt2017matrix}, p.\, 51) are competitive and among the most efficient.

Product integrals may also be used in the construction of time--inhomogeneous Markov processes if the 
primitive is the intensity matrices. Indeed, if $\mat{A}(x)$ are intensity matrices (i.e. off diagonal elements are non--negative and rows sum to 0), then their product integrals are transition matrices, and by \eqref{eq:product-rule}, Chapman--Kolmogorov' equations are then satisfied which implies the Markov property. For further details we refer to \cite{johansen-1986}. 

% Background on methods for calculating reserves, Thiele's differential equations and ramifications. Mogens S. can excel here. Inhomogeneous Markov processes and product integrals, perhaps Mogens B. Matrix notation and vectors (bold capital, their elements as minuscule versions of the same letter,  $\vect{e}$ etc. ) Cite Morris \cite{Morris1982}...
\section{The general model}\label{sec:model}
Consider a
time--in\-ho\-mo\-geneous Markov process $\{ Z(t)\}_{t\geq 0}$ with state space $E=\{1,2,...,p\}$ and intensity matrices $\mat{M}(t)=\{ \mu_{ij}(t)\}_{i,j\in E}$. Let $\mat{P}(s,t) = \{ p_{ij}(s,t) \}_{i,j=1,..,p}$ denote the corresponding transition matrix. Assume that 
\[ \mat{M}(s) = \mat{C}(s)+\mat{D}(s) , \]
where $\mat{D}(s)=\{ d_{ij}(s)\}_{i,j\in E}$ denote a $p\times p$ matrix with $d_{ij}(s)\geq 0$ and $\mat{C}(s)=\{ c_{ij}(s)\}_{i,j\in E}$ is a sub--intensity matrix, i.e. its rows sums are non--positive. We define a reward structure on the Markov process in the following way. At jumps from $i$ to $j$, lump sums of $b^{ij}(s)$ are obtained with probability $d_{ij}(s)/(d_{ij}(s)+c_{ij}(s))$. When $Z(s)=i$ there may be two kind of rewards: a continuous rate of $b^i(s)$, so that $b^i(s)\,\dd s$ is earned during $[s,s+\dd s)$, and lump sums of $b^{ii}(s)$ at the events of a Poisson process with rate $d_{ii}(s)$  while in state $i$. The total reward obtained during $[s,t]$ is then given by
\begin{align}
    R^T(s,t) \ &=\ \sum_i \bar{R}^i(s,t) + \sum_{i\neq j} \bar{N}^{ij}(s,t) \label{eq:def:total-reward} \\
\intertext{where}   
 \bar{N}^{ij}(s,t)\  &=\ \int_s^t b^{ij}(x)\,\dd N^{ij}(x) \label{def:Nbar}, \\
\bar{R}^i(s,t)\ &=\ \int_s^t b^i(x)1\{ Z(x)=i\}\,\dd x \label{def:Rbar}
\end{align}
and where $N^{ij}(x)$ is the counting process which increases by $+1$ upon transition from $i$ to $j$ in $Z(t)$ for $i\neq j$, or a Poisson process with rate $d_{ii}(t)$ if $i=j$.

Our principal application is to life insurance, where the states $i\in E$ are the different conditions of an insured individual (e.g. active, unemployed, disabled or dead). Here we are interested in studying the discounted rewards, $U(s,t)$, during a time interval $[s,t]$ defined by
\[  U(s,t) = \int_s^t \e^{-\int_s^u r(x)\,\dd x} \,\dd B(u)  , \]
where $r(x)$ is a deterministic (instantaneous) interest rate at time $x$ and $B$ is a payment process
\begin{eqnarray*}
\dd B (t)&=&b^{Z(t)}(t)\,\dd t + \sum_{j=1}^p b^{Z(t-)j}(t)\,\dd N^{Z(t-)j}(t) . \label{def:payment-process}
%&=& \sum_i 1\{ Z(t) =i \}dB^i(t) + \sum_{i=1}^p \sum_{j\neq i} b^{ij}(t)dN^j(t) ,
\end{eqnarray*}
Here the continuous rates may e.g.\ be premiums (negative) or compensations (periodic unemployment payments). Lump sums $b^{ij}(t)$ will be paid out at transitions $i$ to $j$ at time $t$ with probability $d_{ij}(t)/(d_{ij}(t)+c_{ij}(t))$, while other lump sums of $b^{ii}(t)$ will be paid out while the insured is in condition $i$ at random times that will appear according to a time--inhomogeneous Poisson process at rate $d_{ii}(t)$.
\begin{remark}
If $c_{ij}(s)=0$ for all $i\neq j$ and $d_{ii}(s)=0$ for all $i$, then we recover the standard multi--state Markov model in life insurance as in e.g. \cite{Hoem1969} or \cite{norberg1991}.
\end{remark}

% \begin{example}
% An example where Poisson arrivals (claims) in certain states could be realistic. This could be a life-insurance policy or pension scheme including some medical care or travel insurance. If a patient is invalid from some diagnosed condition, then no travel insurance would cover attentions (emergencies) relating this condition. Can random sized claims be written as deterministic sized (time dependent) claims by appropriately choosing the rates?
% \end{example}

% \begin{example}
% An example where lump sums at transitions are not necessarily paid out with probability one. 
% This could be transitions to disability where only a lumps sum is paid out if it e.g. happens by accident. This would normally involve to states, but here only one is needed. Maybe free policy modelling could also be modelled with fewer states in this way. 
% \end{example}

% \begin{example}
% Example of cases where the partial reserve could be of interest rather than the reserve itself.
% \end{example}

We are interested in calculating the moments of $R^T(s,t)$ and more generally of $U(s,t)$. To this end we define the slightly more general quantities
\begin{equation}
m_{ij}^{(k)}\  =\  \Exp\Bigl[1\{ Z(t)=j \} R^T(s,t)^k \,\Big|\, Z(s)=i \Bigr]
\end{equation}
and
\begin{equation}
  \mat{m}^{(k)}(s,t) = \big\{m_{ij}^{(k)} \big\}_{i,j\in E}   \label{eq:def-m}
\end{equation}
and more generally,
 \begin{equation}
   v_{ij}^{(k)}(s,t)\ =\ \Exp\Bigl[ 1\{ Z(t)=j\} U(s,t)^k  \,\Big|\, Z(s)=i  \Bigr] \label{eq:moments-future-payments}
\end{equation}
and 
 \begin{equation}
  \mat{V}^{(k)}(s,t)\ =\ \bigl\{  v_{ij}^{(k)}(s,t) \bigr\}_{i,j\in E}   \label{eq:moments-future-payments-matrix}
  \end{equation}
for $k\in \mathbb{N}$.
Define
\begin{eqnarray}
 \mat{B}(t) &=& \{ b^{ij}(t) \}_{i,j\in E} \\
 \vect{b}(t)&=&(b^1(t),...,b^{p}(t))^\prime \\
 \mat{R}(t)&=&\mat{D}(t)\bullet \mat{B}(t) + \mat{\Delta}(\vect{b}(t)) \\
 \mat{C}^{(k)}(t)&=&\mat{D}(t)\bullet \mat{B}^{\bullet k}(t),\ \   k\geq 2 \label{def-symbol:C}
\end{eqnarray}
where $\mat{\Delta}(\vect{b}(t))$ denotes the diagonal matrix with the vector $\vect{b}(t)$ as diagonal, $\bullet$
is the Schur (entrywise) matrix product, i.e. $\{ a_{ij} \}\bullet \{ b_{ij} \} = \{ a_{ij}b_{ij} \}$ and $\mat{B}^{\bullet k}(t) = \mat{B}(t)\bullet \cdots \bullet \mat{B}(t)$ ($n$ terms).

The state--wise prospective reserves are defined as $\Exp [U(s,t)|Z(s)=i]$ for all $i\in E$, which are then the elements of the vector $\mat{V}^{(1)}(s,t)\vect{e}$,  where
$\vect{e}$ is the column vector of ones (see \eqref{eq:moments-future-payments-matrix}). We shall say that the matrix $\mat{V}^{(1)}(s,t)$ contains the {\it partial (state--wise prospective) reserves} and the refer to matrix itself as such. Though the partial reserve may have its own merit, it is introduced primarily for mathematical convenience. 

\section{Partial reserves and Thiele's differential equations}\label{sec:Thiele}
First we start with an integral representation of the first order moment $ \mat{m}^{(1)}(s,t)$.
\begin{lemma}\label{lemma:rewards}
$\displaystyle \mat{m}^{(1)}(s,t) \ =\  \int_s^t \mat{P}(s,u)\mat{R}(u)\mat{P}(u,t)\,\dd u $\,.
\end{lemma}
\begin{proof}
The probability that there is a jump from $k$ to $\ell$ in $[u,u+\dd u)$ given that $Z(s)=i$ is
\begin{eqnarray*}
p_{ik}(s,u)\mu_{k\ell}(u)\,\dd u &=& 
\vect{e}_i^\prime \mat{P}(s,u)\vect{e}_k \mu_{k\ell}(u)\,\dd u . 
\end{eqnarray*}
Hence, the probability that there is a jump from $k$ to $\ell$ in $[u,u+\dd u)$ and that $Z(t)=j$, given that $Z(s)=i$, is
consequently 
\begin{eqnarray*}
p_{ik}(s,u)\mu_{k\ell}(u)\,\dd u p_{\ell j}(u,t) &=& 
\vect{e}_i^\prime \mat{P}(s,u)\vect{e}_k \mu_{k\ell}(u)\,\dd u \vect{e}_\ell^\prime \mat{P}(u,t) \vect{e}_j. 
\end{eqnarray*}
This serves as the density for a jump at $u$. The reward $b^{k\ell}(u)$ is the value to be paid out with probability
\[  \frac{d_{k\ell}(u)}{d_{k\ell}(u)+c_{k\ell}(u)} = \frac{d_{k\ell}(u)}{\mu_{k\ell}(u)} . \]
Hence the expected reward in $[s,t]$ due to jumps from $k$ to $\ell$ amount to
\[  \int_s^t \vect{e}_i^\prime \mat{P}(s,u)\vect{e}_k \mu_{k\ell}(u) \frac{d_{k\ell}(u)}{\mu_{k\ell}(u)} b^{k\ell}(u)  \vect{e}_\ell^\prime  \mat{P}(u,t) \vect{e}_j \,\dd u  = \int_s^t \vect{e}_i^\prime \mat{P}(s,u)\vect{e}_k d_{k\ell}(u)  b^{k\ell}(u)  \vect{e}_\ell^\prime  \mat{P}(u,t) \vect{e}_j \,\dd u . \]
This settles the case for the jumps. Concerning the linear rewards, consider state $k$. Here its contribution amounts to
\begin{eqnarray*}
  \lefteqn{\Exp\left[ \int_s^t\left. 1\{ Z(t)=j \} b^k(u)1\{ Z(u)=k \}\,\dd s \right| Z(s)=i \right]}~~~~~~~~~~~~~~~~~~\\
  &=& 
  \int_s^t b^{k}(u) \Prob (Z(t)=j, Z(u)=k | Z(s)=i)\,\dd s \\
  &=& \int_s^t  b^{k}(u)\vect{e}_i^\prime \mat{P}(s,u)\vect{e}_k  \vect{e}_k^\prime \mat{P}(u,t)\vect{e}_j \,\dd s .
  \end{eqnarray*} 
For the case of lump sums from Poisson arrivals in state $k$, the contribution is
 \[ \int_s^t \vect{e}_i^\prime \mat{P}(s,u)\vect{e}_k d_{kk}(u)  b^{k\ell}(u)  \vect{e}_k^\prime  \mat{P}(u,t) \vect{e}_j \,\dd u .\] 
The total reward is then obtained by summing over all three different types of reward, which in matrix notation exactly yields the result.
\end{proof}
Next we consider the partial reserve. For a fixed time horizon (the terminal date) $T$ we define
 \begin{equation}
  \mat{V}(t) = \mat{V}^{(1)}(t,T) .  \label{def:partial-reserve-1}
\end{equation}
From Lemma \ref{lemma:rewards} we  have that
\[  \Exp \bigl[\dd B(u)1\{ Z(T)=j \} \,\big|\, Z(t)=i\bigr]\ =\ \vect{e}_i^\prime \mat{P}(t,u)\mat{R}(u)\mat{P}(u,T)\vect{e}_j \,\dd u .\]
We denote the elements of $\mat{V}(t)$ by $v_{ij}(t)$. Then we have the following Thiele type of differential equation for the partial reserve.
\begin{theorem}[Thiele]\label{th:thiele-partial}
$\mat{V}(t)=\mat{V}^{(1)}(t,T)$ satisfies
\[ \frac{\partial}{\partial t}  \mat{V}(t) = r(t) \vect{V}(t)-\mat{M}(t) \vect{V}(t) - \mat{R}(t)\mat{P}(t,T)  \]
with terminal condition $\mat{V}(T)=\mat{0}$.
\end{theorem}
\begin{proof}
First we see that
\begin{eqnarray}
  v_{ij}(t) &=&\int_t^T \e^{-\int_s^u r(u)\,\dd u} \Exp \bigl[1\{ Z(T)=j \} \,\dd B(u) \,\big|\,  Z(t)=i \bigr] \nonumber \\
  &=&\int_t^T \e^{-\int_t^u r(s)\,\dd s} \vect{e}_i^\prime \mat{P}(t,u)\mat{R}(u)\mat{P}(u,T)\vect{e}_j \,\dd u\nonumber \\
  &=&\vect{e}_i^\prime \int_t^T \prod_t^u (1 -r(s)\,\dd s) \prod_t^u (\mat{I} + \mat{M}(s)\,\dd s) \mat{R}(u) \prod_u^T (\mat{I} + \mat{M}(s)\,\dd s))\,\dd u \vect{e}_j \nonumber\\
 % &=& \int_t^T \e^{-\int_t^u r(s)ds} \vect{e}_i^\prime \mat{P}(t,u)\mat{R}(u)\mat{P}(u,T)\vect{e} du \\
%  &=&\vect{e}_i^\prime \int_t^T \prod_t^u (1 -r(s)ds) \prod_t^u (\mat{I} + \mat{\Lambda}(s)ds) \mat{R}(u) \prod_u^T (\mat{I} + \mat{\Lambda}(s)ds))du \vect{e} \\
  &=& \vect{e}_i^\prime \int_t^T  \prod_t^u (\mat{I} + (\mat{M}(s)-r(s)\mat{I})\,\dd s) \mat{R}(u) \prod_u^T (\mat{I} + \mat{M}(s)\,\dd s))\,\dd u \vect{e}_j . \nonumber
\end{eqnarray} 
In matrix notation,
\begin{equation}
 \mat{V}(t) = \int_t^T  \prod_t^u (\mat{I} + (\mat{M}(s)-r(s)\mat{I})\,\dd s) \mat{R}(u) \prod_u^T (\mat{I} + \mat{M}(s)\,\dd s))\,\dd u . \label{eq:extend-reserve-int-rep}
 \end{equation}
Thus
\begin{eqnarray*}
\lefteqn{\frac{\partial}{\partial t}  \mat{V}(t)=\frac{\partial}{\partial t} \int_t^T  \prod_t^u (\mat{I} + (\mat{M}(s)-r(s)\mat{I})\,\dd s) \mat{R}(u) \prod_u^T (\mat{I} + \mat{M}(s)\,\dd s))\,\dd u }~~~~\\
&=& \int_t^T \left[\frac{\partial}{\partial t} \prod_t^u (\mat{I} + (\mat{M}(s)-r(s)\mat{I})\,\dd s)\right] \mat{R}(u) \prod_u^T (\mat{I} + \mat{M}(s)\,\dd s))\,\dd u \\
&&- \mat{I}\cdot \mat{R}(t) \prod_t^T (\mat{I} + \mat{M}(s)\,\dd s)) \\
&=& - (\mat{M}(t)-r(t)\mat{I})) \int_t^T  \prod_t^u (\mat{I} + (\mat{M}(s)-r(s)\mat{I})\,\dd s) \mat{R}(u) \prod_u^T (\mat{I} + \mat{M}(s)\,\dd s))\,\dd u \\
&&- \mat{R}(t) \prod_t^T (\mat{I} + \mat{M}(s)\,\dd s)) \\
&=&- (\mat{M}(t)-r(t)\mat{I})) \vect{V}(t) - \mat{R}(t) \prod_t^T (\mat{I} + \mat{M}(s)\,\dd s)) \\
&=&- (\mat{M}(t)-r(t)\mat{I})) \vect{V}(t) -  \mat{R}(t)\mat{P}(t,T) ,
\end{eqnarray*}
i.e.
\begin{equation}
\frac{\partial}{\partial t}  \mat{V}(t) =r(t)\vect{V}(t)- \mat{M}(t)\vect{V}(t) -  \mat{R}(t)\mat{P}(t,T),
 \label{eq:ext-thiele}
\end{equation} 
with obvious boundary condition $\mat{V}(T)=\mat{0}$.
\end{proof}
As an immediate consequence, using that $\mat{P}(t,T)\vect{e} =\vect{e}$
being a transition probability matrix, we recover the usual Thiele differential equation.
\begin{corollary}
The vector of prospective reserves $ \vect{V}_{th}(t)=\mat{V}^{(1)}(t,T)\vect{e}$ satisfies 
\begin{eqnarray*}
\frac{\partial}{\partial t}\vect{V}_{th}(t)
&=&r(t) \vect{V}_{th}(t) - \mat{M}(t) \vect{V}_{th}(t) - \mat{R}(t)\vect{e}, \label{eq:thiele}
\end{eqnarray*}
with terminal condition $ \vect{V}_{th}(T)=\vect{0}$.
\end{corollary}
\section{Matrix representation of the reserve}\label{sec:matrix-reserves}
In this section we will provide a matrix representation of the reserve. We start with an important general result which extends a result by \cite{VanLoan:1978tq} from matrix--exponentials to product integrals. Here
the matrices $\mat{A}(x)$ and $\mat{C}(x)$ must be square but not necessarily of the same dimension.
\begin{lemma}\label{lemma:van-loan}
For matrix functions $\mat{A}(x)$, $\mat{B}(x)$ and $\mat{C}(x)$, we have that
\[  \prod_s^t \left(  \mat{I} + 
\begin{pmatrix}
\mat{A}(x) & \mat{B}(x) \\
\mat{0} & \mat{C}(x)
\end{pmatrix} \,\dd x
\right) 
=
\begin{pmatrix}
\displaystyle\prod_s^t (\mat{I}+\mat{A}(x)\,\dd x) & \displaystyle \int_s^t \prod_s^u (\mat{I}+\mat{A}(x)\,\dd x)\mat{B}(u)\prod_u^t (\mat{I}+\mat{C}(x)\,\dd x)\,\dd u \\
\mat{0} & \displaystyle\prod_s^t (\mat{I}+\mat{C}(x)\,\dd x)
\end{pmatrix} .
 \]
 
\end{lemma}
\begin{proof}
First notice that
\begin{eqnarray*}\lefteqn{
\frac{\partial}{\partial t}\int_s^t \prod_s^u (\mat{I}+\mat{A}(x)\,\dd x)\mat{B}(u)\prod_u^t (\mat{I}+\mat{C}(x)\,\dd x)\,\dd u}~~~\\
&=&
\int_s^t \prod_s^u (\mat{I}+\mat{A}(x)\,\dd x)\mat{B}(u)\frac{\partial}{\partial t}\prod_u^t (\mat{I}+\mat{C}(x)\,\dd x)\,\dd u 
+  \prod_s^t (\mat{I}+\mat{A}(u))\mat{B}(t) \prod_t^t (\mat{I}+\mat{C}(x)\,\dd x) \\
&=& \int_s^t \prod_s^u (\mat{I}+\mat{A}(x)\,\dd x)\mat{B}(u)\prod_u^t (\mat{I}+\mat{C}(x)\,\dd x)\,\dd u\ \mat{C}(t) 
+ \prod_s^t (\mat{I}+\mat{A}(u))\mat{B}(t)
\end{eqnarray*}
Let $\mat{B}(s,t)$ denote the matrix on the right hand side in the Lemma. Then
\begin{eqnarray*}
\frac{\partial}{\partial t}\mat{B}(s,t)&=& 
\begin{pmatrix}
\displaystyle\frac{\partial}{\partial t} \prod_s^t (\mat{I}+\mat{A}(x)\,\dd x) & \displaystyle\frac{\partial}{\partial t}\int_s^t \prod_s^u (\mat{I}+\mat{A}(x)\,\dd x)\mat{B}(u)\prod_u^t (\mat{I}+\mat{C}(x)\,\dd x)\,\dd u \\
\mat{0} & \displaystyle\frac{\partial}{\partial t} \prod_s^t (\mat{I}+\mat{C}(x)\,\dd x) 
\end{pmatrix} \\
&=& \begin{pmatrix}
\displaystyle\prod_s^t (\mat{I}+\mat{A}(x)\,\dd x) & \displaystyle \int_s^t \prod_s^u (\mat{I}+\mat{A}(x)\,\dd x)\mat{B}(u)\prod_u^t (\mat{I}+\mat{C}(x)\,\dd x)\,\dd u \\
\mat{0} & \displaystyle\prod_s^t (\mat{I}+\mat{C}(x)\,\dd x)
\end{pmatrix}
\begin{pmatrix}
\mat{A}(t) & \mat{B}(t) \\
\mat{0} & \mat{C}(t)
\end{pmatrix} \\
&=& \mat{B}(s,t) \begin{pmatrix}
\mat{A}(t) & \mat{B}(t) \\
\mat{0} & \mat{C}(t)
\end{pmatrix} .
\end{eqnarray*}
Also, $\mat{B}(t,t)=\mat{I}$. Therefore, 
\[ \mat{B}(s,t) = \prod_s^t \left( \mat{I}+ \begin{pmatrix}
\mat{A}(u) & \mat{B}(u) \\
\mat{0} & \mat{C}(u)
\end{pmatrix}\,\dd u\right) . \]
\end{proof} 
\noindent Consider an $np\times np$ block matrix on the form
\[ \mat{A}(x) = \begin{pmatrix}
\mat{A}_{11}(x) & \mat{A}_{12}(x) & \cdots & \mat{A}_{1n}(x) \\
\mat{0} & \mat{A}_{22}(x) & \cdots & \mat{A}_{2n}(x) \\
\mat{0} & \mat{0} & \cdots & \mat{A}_{3n}(x) \\
\vdots & \vdots & \vdots\vdots\vdots & \vdots \\
 \mat{0} & \mat{0} & \cdots & \mat{A}_{nn}(x) 
\end{pmatrix}  \]
and write
\[ \mat{A}(x) =\prod_s^t (\mat{I} + \mat{A}(x)\,\dd x) 
 = \begin{pmatrix}
\mat{B}_{11}(s,t) & \mat{B}_{12}(s,t) & \cdots & \mat{B}_{1n}(s,t) \\
\mat{0} & \mat{B}_{22}(s,t) & \cdots & \mat{B}_{2n}(s,t) \\
\mat{0} & \mat{0} & \cdots & \mat{B}_{3n}(s,t) \\
\vdots & \vdots & \vdots\vdots\vdots & \vdots \\
 \mat{0} & \mat{0} & \cdots & \mat{B}_{nn}(s,t) 
\end{pmatrix} . \]
Then Lemma \ref{lemma:van-loan} implies that
  \begin{eqnarray*}
  \lefteqn{\big( \mat{B}_{12}(s,t),\mat{B}_{13}(s,t),...,\mat{B}_{1n}(s,t) \big)=}~~~\\
  && \int_s^t \prod_s^x (\mat{I}+\mat{A}_{11}(u)\,\dd u) \left[ \mat{A}_{12}(x),...,\mat{A}_{1n}(x) \right]
  \prod_x^t \left(\mat{I} + 
  \begin{pmatrix}
  \mat{A}_{22}(u) & \mat{A}_{23}(u) & \cdots & \mat{A}_{2n}(u) \\
  \mat{0}& \mat{A}_{33}(u) & \cdots & \mat{A}_{3n}(u) \\
  \mat{0} & \mat{0} & \cdots & \mat{A}_{4n}(u) \\
  \vdots & \vdots & \vdots\vdots\vdots & \vdots \\
   \mat{0} & \mat{0} & \cdots & \mat{A}_{nn}(u) 
  \end{pmatrix} \,\dd u
  \right)\end{eqnarray*}
  so that
  \[ \mat{B}_{1n}(s,t) = \int_s^t \prod_s^x (\mat{I}+\mat{A}_{11}(u)\,\dd u) \left[ \mat{A}_{12}(x),...,\mat{A}_{1n}(x) \right]
\begin{pmatrix}
\mat{B}_{2n}(x,t)\\
\mat{B}_{3n}(x,t)\\
\vdots \\
\mat{B}_{nn}(x,t)
\end{pmatrix}\ \,\dd x
   \]
 which can be written as
  \[ \mat{B}_{1n}(s,t) = \sum_{i=2}^n \int_s^t \prod_s^x (\mat{I}+\mat{A}_{11}(u)\,\dd u) \mat{A}_{1i}(x)\mat{B}_{in}(x,t)\,\dd x .\]
  Applying Lemma \ref{lemma:van-loan} to  \label{eq:extend-reserve} from downwards and up, we then get that
  \begin{eqnarray*}
  \mat{B}_{nn}(s,t)&=&\prod_s^t (\mat{I}+\mat{A}_{nn}(x)\,\dd x) \\
  \mat{B}_{n-1,n}(s,t)&=&\int_s^t \prod_s^x (\mat{I}+\mat{A}_{n-1,n-1}(u)\,\dd u)\mat{A}_{n-1,n}(x)\prod_x^t (\mat{I}+\mat{A}_{nn}(u)\,\dd u)\\
  &=&\int_s^t \prod_s^x (\mat{I}+\mat{A}_{n-1,n-1}(u)\,\dd u)\mat{A}_{n-1,n}(x)\mat{B}_{nn}(x,t)\,\dd x
  \end{eqnarray*}
  and more generally,
\begin{equation}
   \mat{B}_{i,n}(s,t) = \sum_{j=i+1}^n \int_s^t \prod_s^x (\mat{I}+\mat{A}_{ii}(u)\,\dd u)\mat{A}_{ij}(x)\mat{B}_{jn}(x,t)\,\dd x . \label{eq:recursion}
\end{equation}
\begin{theorem}
  The partial reserve defined in \eqref{def:partial-reserve-1}, $\mat{V}(t)$, (and the transition matrix) can be calculated through the product integral
  \begin{eqnarray*}
\prod_t^T \left( \mat{I} + 
\begin{pmatrix}
\mat{M}(u)-r(u)\mat{I} & \mat{R}(u) \\
\mat{0} & \mat{M}(u)
\end{pmatrix} \,\dd u
 \right)
 &=&\begin{pmatrix}
 \displaystyle\prod_t^T (\mat{I} + (\mat{M}(u)-r(u)\mat{I})\,\dd u) & \mat{V}(t) \\
 \mat{0} & \mat{P}(t,T) 
 \end{pmatrix}   .
\end{eqnarray*}
  \end{theorem}
\begin{proof}
Applying Lemma \ref{lemma:van-loan} to \eqref{eq:extend-reserve-int-rep} we get that
\begin{eqnarray}
\lefteqn{
\prod_t^T \left( \mat{I} + 
\begin{pmatrix}
\mat{M}(u)-r(u)\mat{I} & \mat{R}(u) \\
\mat{0} & \mat{M}(u)
\end{pmatrix} \,\dd u
 \right)} \nonumber \\
 &=& 
 \begin{pmatrix}
 \displaystyle\prod_t^T (\mat{I} + (\mat{M}(u)-r(u)\mat{I})\,\dd u) & \displaystyle\int_t^T  \prod_t^u (\mat{I} + (\mat{M}(s)-r(s)\mat{I})\,\dd s) \mat{R}(u) \prod_u^T (\mat{I} + \mat{M}(s)\,\dd s))\,\dd u \\
 \mat{0} & \displaystyle\prod_t^T (\mat{I} + \mat{M}(u)\,\dd u) .
 \end{pmatrix}\nonumber \\
 &=&\begin{pmatrix}
 \displaystyle\prod_t^T (\mat{I} + (\mat{M}(u)-r(u)\mat{I})\,\dd u) & \mat{V}(t) \\
 \mat{0} & \mat{P}(t,T) .
 \end{pmatrix}  . \label{eq:Thiele-explicit-sol}
\end{eqnarray}\end{proof}
So both the partial reserve $\mat{V}(t)$ and $\mat{P}(t,T)=\{ p_{ij}(t,T)\}$ are calculated trough one single calculation of the product integral. This is convenient if we are interested in the calculating of the expected future payments conditional on $Z(t)=i$ and $Z(T)=j$, since
\[ \Exp \bigl[B(T)\,\big|\, Z(0)=i,Z(T)=j\bigr]\ =\ \frac{v_{ij}(t)}{p_{ij}(t,T)} . \] 
 % One way of solving for the $\mat{V}(t)$ and $\mat{P}(t,T)$ is through the extended system of differential equation (satisfied by any product integral). 

\section{Laplace transform of rewards and future payments}\label{sec:laplace-transform}
Recall the definition \eqref{eq:def:total-reward} of the total reward $R^T(s,t)$ which is the undiscounted future payments in an insurance context.  
Let
\[ F_{ij}(x;s,t) = \Prob (Z(t)=j,  R^T(s,t)\leq x \ | \ Z(s)=i)  \]
and
\[  F^*_{ij}(\theta;s,t) = \int_{-\infty}^\infty \e^{-\theta x}\,\dd F_{ij}(x;s,t)  \]
which is assumed to exist. Define 
\[   \mat{F}^*(\theta;s,t) =  \{  F^*_{ij}(\theta;s,t) \}_{i,j=1,...,p} .\]
\begin{theorem}\label{Th:4.5a}
 The distribution of the total reward $R^T(s,t)$ has Laplace--Stieltjes transform given by 
\[ \mat{F}^*(\theta;s,t) = \prod_{s}^t \left( \mat{I} + \left[ \mat{D}(u)\bullet \big\{ \e^{-\theta b^{k\ell}(u)} \big\}_{k,\ell}+\mat{C}(u) - \theta \mat{\Delta}(\vect{b}(u)) \right] \,\dd u \right) .  \]
 \end{theorem}
\begin{proof}
Conditioning on the time $u$ of the first jump, if any, in $[s,t]$, we get
\begin{eqnarray*}
F_{ij}(x;s,t)&=&\delta_{ij}\exp \left( \int_s^t \mu_{ii}(u)\,\dd u \right)1\left\{ \int_s^t b^i(u)\,\dd u\leq  x \right\} \\
&&+ \sum_{k=1}^p \int_s^t \exp \left( \int_s^u \mu_{ii}(r)\,\dd dr \right)d_{ik}(u)F_{kj}\left(x-\int_s^ub^i(r)dr-b^{ik}(u);u,t\right) \,\dd u  \\
&&+ \sum_{k\neq p} \int_s^t \exp \left( \int_s^u \mu_{ii}(r)\,\dd r \right)c_{ik}(u)F_{kj}\left(x-\int_s^ub^i(r)\,\dd r;u,t\right) \,\dd u . 
\end{eqnarray*}
Here the first line corresponds to the case where there are no Poisson arrivals or jumps in $[s,t]$, the second line corresponds to a jump with reward for $k\neq i$ or a Poisson arrival rewards for $k=i$, while the third line corresponds to a jump without rewards. Consider the Laplace transform
\[  F^0_{ij}(\theta ; s,t)= \int_{-\infty}^\infty \e^{-\theta x}F_{ij}(x;s,t)\,\dd x  . \]
 Then the contribution from the first term amounts to
 \begin{eqnarray*}
 \int_{-\infty}^\infty \e^{-\theta x} \delta_{ij}\exp \left( \int_s^t \mu_{ii}(u)\,\dd u \right)1\left\{ \int_s^t b^i(u)\,\dd u\leq  x \right\} \,\dd x &=&
 \delta_{ij}\exp \left( \int_s^t \mu_{ii}(u)\,\dd u \right)\int_{\int_s^t b^i(u)\,\dd u}^\infty \e^{-\theta x} \,\dd x\\
 &=&\frac{\delta_{ij}}{\theta}\exp \left( \int_s^t \mu_{ii}(u)\,\dd u - \theta\int_s^t b^i(u)\,\dd u \right) .
 \end{eqnarray*}
The second term contributes to the Laplace transform by 
 \begin{eqnarray*}
 \lefteqn{\int_{-\infty}^\infty \e^{-\theta x}\sum_{k=1}^p\int_s^t \exp \left( \int_s^u \mu_{ii}(r)\,\dd r \right)d_{ik}(u)F_{kj}\left(x-\int_s^ub^i(r)dr-b^{ik}(u);u,t\right) \,\dd u\ \,\dd x}~~\\
 &=&\sum_{k=1}^p\int_s^t \exp \left( \int_s^u \mu_{ii}(r)\,\dd r \right)d_{ik}(u) \int_{-\infty}^\infty \e^{-\theta x} F_{kj}\left(x-\int_s^ub^i(r)\,\dd r-b^{ik}(u);u,t\right)\,\dd x \ \,\dd u \\
 &=&\sum_{k=1}^p\int_s^t \exp \left( \int_s^u \mu_{ii}(r)\,\dd r \right)d_{ik}(u) 
\exp \left( -\theta\int_s^ub^i(r)\,\dd r -\theta b^{ik}(u)  \right)
 \int_{-\infty}^\infty \e^{-\theta x} F_{kj}\left(x;u,t\right)\,\dd x \ \,\dd u \\
 &=&\sum_{k=1}^p\int_s^t \exp \left( \int_s^u \mu_{ii}(r)\,\dd r \right)d_{ik}(u)
 \exp \left( -\theta \int_s^ub^i(r)\,\dd r-\theta b^{ik}(u) \right) F^0_{kj}(\theta;u,t) \,\dd u .
 \end{eqnarray*}
 The third term contributes similarly by 
\begin{eqnarray*}
 \lefteqn{\int_{-\infty}^\infty \e^{-\theta x}\sum_{k\neq i}\int_s^t \exp \left( \int_s^u \mu_{ii}(r)\,\dd r \right)c_{ik}(u)F_{kj}\left(x-\int_s^ub^i(r)dr;u,t\right) \,\dd u\ \,\dd x}~~\\
 &=&\sum_{k\neq i}\int_s^t \exp \left( \int_s^u \mu_{ii}(r)\,\dd r \right)c_{ik}(u)
 \exp \left( -\theta \int_s^u b^i(r)\,\dd r \right) F^0_{kj}(\theta;u,t) \,\dd u .
 \end{eqnarray*}
 Thus
 \begin{eqnarray*}
 F^0_{ij}(\theta;s,t)&=&\frac{\delta_{ij}}{\theta}\exp \left( \int_s^t \mu_{ii}(r)\,\dd r - \theta \int_s^t b^i(r)\,\dd r \right) \\
&&+\sum_{k=1}^p\int_s^t \exp \left( \int_s^u \mu_{ii}(r)\,\dd r \right)d_{ik}(u)
 \exp \left( -\theta \int_s^ub^i(r)\,\dd r-\theta b^{ik}(u) \right) F^0_{kj}(\theta;u,t) \,\dd u \\
 &&+ \sum_{k\neq i}\int_s^t \exp \left( \int_s^u \mu_{ii}(r)\,\dd r \right)c_{ik}(u)
 \exp \left( -\theta \int_s^u b^i(r)\,\dd r \right) F^0_{kj}(\theta;u,t) \,\dd u .
 \end{eqnarray*}
 For the corresponding Laplace--Stieltjes transform
 \[  F^*_{ij}(\theta;s,t) = \theta F^0_{ij}(\theta;s,t)  \]
 we then get
 \begin{eqnarray*}
 F^*_{ij}(\theta;s,t)&=&\delta_{ij}\exp \left( \int_s^t \mu_{ii}(u)\,\dd u - \theta \int_s^t b^i(r)\,\dd r \right) \\
&&+\sum_{k=1}^p\int_s^t \exp \left( \int_s^u \mu_{ii}(r)\,\dd r \right)d_{ik}(u)
 \exp \left( -\theta \int_s^ub^i(r)\,\dd r-\theta b^{ik}(u) \right) F^*_{kj}(\theta;u,t) \,\dd u \\
 &&+ \sum_{k\neq i}\int_s^t \exp \left( \int_s^u \mu_{ii}(r)\,\dd r \right)c_{ik}(u)
 \exp \left( -\theta \int_s^u b^i(r)\,\dd r \right) F^*_{kj}(\theta;u,t) \,\dd u 
 \end{eqnarray*}
 so
 \begin{eqnarray*}
 F^*_{ij}(\theta;s,t)\exp \left( -\int_s^t \mu_{ii}(u)\,\dd u + \theta \int_s^t b^i(r)\,\dd r \right)\\&&
 \hspace{-6cm}=\delta_{ij} 
+\sum_{k=1}^p\int_s^t \exp \left( -\int_u^t \mu_{ii}(r)\,\dd r \right)d_{ik}(u)
 \exp \left( \theta \int_u^t b^i(r)\,\dd r-\theta b^{ik}(u) \right) F^*_{kj}(\theta;u,t) \,\dd u \\
 &&\hspace{-6cm}+ \sum_{k\neq i}\int_s^t \exp \left( -\int_u^t \mu_{ii}(r)\,\dd r \right)c_{ik}(u)
 \exp \left( \theta \int_u^t b^i(r)\,\dd r \right) F^*_{kj}(\theta;u,t) \,\dd u 
 \end{eqnarray*}
 Now differentiate with respect to $s$ to get
  \begin{eqnarray*}
\lefteqn{\frac{\partial F^*_{ij}}{\partial s} (\theta;s,t)\exp \left( -\int_s^t \mu_{ii}(u)\,\dd u + \theta \int_s^t b^i(r)\,\dd r \right)}~~~\\
&&+ F^*_{ij}(\theta;s,t)\exp \left( -\int_s^t \mu_{ii}(u)\,\dd u + \theta \int_s^t b^i(r)\,\dd r \right) (\mu_{ii}(s)-\theta b^i(s))  \\&&
=- \sum_{k=1}^p \exp \left( -\int_s^t \mu_{ii}(r)dr \right)d_{ik}(s)
 \exp \left( \theta \int_s^t b^i(r)\,\dd r-\theta b^{ik}(s) \right) F^*_{kj}(\theta;s,t)  \\
 &&- \sum_{k\neq i}\exp \left( -\int_s^t \mu_{ii}(r)\,\dd r \right)c_{ik}(s)
 \exp \left( \theta \int_s^t b^i(r)\,\dd r \right) F^*_{kj}(\theta;s,t)
 \end{eqnarray*}
 which implies that
 \begin{eqnarray*}
\frac{\partial F^*_{ij}}{\partial s} (\theta;s,t)&=&
-\sum_{k} \left[ \mat{D}(s) \bullet \{ \e^{-\theta b^{mn}(s)}\}_{m,n} + \mat{C}(s) - \theta \mat{\Delta}(\vect{b}(s)) \}  \right]_{ik}
  F^*_{kj}(\theta;s,t) ,
 \end{eqnarray*}
 or, in matrix notation,
\begin{equation}
  \frac{\partial}{\partial s} \mat{F}^*(\theta;s,t) = -
 \left[ \mat{D}(s) \bullet \{ \e^{-\theta b^{mn}(s)} \}_{m,n} + \mat{C}(s)- \theta \mat{\Delta}(\vect{b}(s)) \}  \right] 
 \mat{F}^*(\theta;s,t) \label{eq:diff-eq-transform}
  \end{equation}
 with boundary condition of $\mat{F}^*(\theta;t,t)=\mat{I}$. Thus the solution is given by 
 \[ \mat{F}^*(\theta;s,t) = \prod_{s}^t \left( \mat{I} + \left[ \mat{D}(u)\bullet \{ \e^{-\theta b^{mn}(u)} \}_{m,n} + \mat{C}(u) -\theta \mat{\Delta}(\vect{b}(u)) \right] \,\dd u \right) .\]
 \end{proof}

\section{Higher order moments}\label{sec:moments}
Define the matrices 
\begin{equation}
\mat{F}^{(k)}(x)=
\arraycolsep=3.8pt\def\arraystretch{2.2}
\left(\begin{array}{ccccccc}
\mat{M}(x) & {k\choose 1}\mat{R}(x) & {k\choose 2}\mat{C}^{(2)}(x) &  \cdots &  {k\choose k-1}\mat{C}^{(k-1)}(x) & \mat{C}^{(k)}(x) \\
\mat{0} & \mat{M}(x)\mat{I} & {k-1 \choose 1}\mat{R}(x) & \cdots & {k-1\choose k-2}\mat{C}^{(k-2)}(x) &\mat{C}^{(k-1)}(x) \\
%\mat{0} &\mat{0} &\mat{M}-(k-2)r(x)\mat{I} & {k-2\choose 1}\mat{R} & \cdots &  {k-2\choose k-3}\mat{C}^{(k-3)}  & \mat{C}^{(k-2)} \\
\vdots & \vdots & \vdots &  \vdots \vdots \vdots & \vdots &\vdots  \\
\mat{0} &\mat{0} &\mat{0} & \cdots & \mat{M}(x) &\mat{R}(x) \\
\mat{0} &\mat{0} &\mat{0} & \cdots  &\mat{0} & \mat{M}(x) 
\end{array}\right) . \label{eq:F-gen-res}
\end{equation}
and 
\begin{eqnarray}
\mat{H}^{(k)}(s,t)=
\arraycolsep=3.8pt\def\arraystretch{2.2}
\left(\begin{array}{ccccccc}
\mat{P}(s,t) & {k\choose 1}\mat{m}^{(1)}(s,t) & {k\choose 2}\mat{m}^{(2)}(s,t) &  \cdots &  {k\choose k-1}\mat{m}^{(k-1)}(s,t) & \mat{m}^{(k)}(s,t) \\
\mat{0} & \mat{P}(s,t) & {k-1 \choose 1}\mat{m}^{(1)}(s,t) & \cdots & {k-1\choose k-2}\mat{m}^{(k-2)}(s,t) &\mat{m}^{(k-1)}(s,t)\\
\vdots & \vdots & \vdots & \vdots \vdots \vdots & \vdots &\vdots  \\
\mat{0} &\mat{0} &\mat{0} & \cdots & \mat{P}(s,t) &\mat{m}^{(1)}(s,t) \\
\mat{0} &\mat{0} &\mat{0} & \cdots  &\mat{0} & \mat{P}(s,t)
\end{array}\right)  . \label{eq:H-moments-of-F}
\end{eqnarray}
Then we have the following main result.
\begin{theorem}\label{th:main-moments}
\[   \prod_s^t (\mat{I} + \mat{F}^{(k)}(x)\,\dd x) =  \mat{H}^{(k)}(s,t) .   \]
\end{theorem}
\begin{proof}
 From the Laplace--Stieltjes transform it is now possible to derive higher order moments. First we notice that $\mat{F}^*(0;s,t)=\prod_s^t (\mat{I}+\mat{M}(u)\,\dd u)$, and we can obtain (recall \eqref{eq:def-m})
  \[  \mat{m}^{(k)}(s,t) \ =\ \Bigl\{ \Exp \Bigl[1\{ Z(t)=j\} R^T(s,t)^k \,\Big|\, Z(s)=i\Bigr]\Bigr\}_{i,j}  \]
  by
  \[  \mat{m}^{(k)}(s,t)= (-1)^k\frac{\partial^k}{\partial \theta^k} \mat{F}^*(\theta;s,t)\biggr\rvert_{\theta=0} . \]
  % Define\footnote{B(t) already used, called it C}
  % \[ \mat{C}(t) = \{ b^{ij}(t) \}_{i,j} . \]
  Now
 \[  \frac{\partial^k}{\partial \theta^k} \{ \e^{-\theta b^{ij}(t)} \}_{i,j}\biggr\rvert_{\theta=0}  = (-\mat{B}(t))^{\bullet k}= (-\mat{B}(t))\bullet \cdots \bullet (-\mat{B}(t)) \]
 ($k$ factors) whereas for $k=0$ (no differentiation, only evaluation at $\theta =0$) it equals the matrix which has all entrances equal to one.
  From \eqref{eq:diff-eq-transform} we get by differentiation with respect to $\theta$ that
 \begin{eqnarray*}
\lefteqn{(-1)^k\frac{\partial^k}{\partial \theta^k}\frac{\partial}{\partial s}\mat{F}^*(\theta;s,t)
 = -(-1)^{k}\frac{\partial^k}{\partial \theta^k}\left(  
  \left[ \mat{D}(s) \bullet \{ \e^{-\theta b^{ij}(s)} \}_{i,j}+\mat{C}(s)- \theta \mat{\Delta}(\vect{b(s)}) \}  \right] \mat{F}^*(\theta;s,t)\right)}~~~ \\
  &=&-\sum_{m=0}^k 
  \begin{pmatrix}
  k \\
  m
  \end{pmatrix}
  (-1)^m\frac{\partial^m}{\partial \theta^m}\left[ \mat{D}(s) \bullet \{ \e^{-\theta b^{ij}(s)} \}_{i,j}+\mat{C}(s)- \theta \mat{\Delta}(\vect{b(s)}) \}  \right] 
  (-1)^{k-m}\frac{\partial^{k-m}}{\partial \theta^{k-m}}\mat{F}^*(\theta;s,t)  .
 \end{eqnarray*}
 Recalling that
 \[ \mat{R}(t) = \mat{D}(t)\bullet \mat{B}(t) + \mat{\Delta}(\vect{b}(t))  \]
 and since 
 \[  \left[ \mat{D}(s) \bullet \{ \e^{-\theta b^{ij}(s)} \}_{i,j}+\mat{C}(s)- \theta \mat{\Delta}(\vect{b(s)}) \}  \right] \biggr\rvert_{\theta=0}  = \mat{D}(s)+\mat{C}(s) = \mat{M}(s)\]
 we get that
 \begin{eqnarray}
 \frac{\partial}{\partial s}\mat{m}^{(k)}(s,t)=-\Big[ \mat{M}(s) \mat{m}^{(k)}(s,t) + 
 k\mat{R}(s)\mat{m}^{(k-1)}(s,t) +\sum_{m=2}^k \begin{pmatrix}
  k \\
  m
  \end{pmatrix}
  \mat{D}(s)\bullet \mat{B}^{\bullet m}(s) \mat{m}^{(k-m)}(s,t)
  \Big],\  \label{eq:moment-differential}
 \end{eqnarray} 
 where
 \[  \mat{m}^{(0)}(s,t) = \mat{F}^*(0;s,t)=\prod_s^t (\mat{I}+\mat{M}(u)\,\dd u) = \mat{P}(s,t) .  \]
 Multiplying from the left on both sides with $\prod_t^s (\mat{I}+\mat{M}(u)\,\dd u)$ (see also \eqref{eq:inverse-prod-int}) we get
 \begin{eqnarray*}
 \frac{\partial}{\partial s}\left(\prod_t^s (\mat{I}+\mat{M}(u)\,\dd u)  \mat{m}^{(k)}(s,t)  \right)&=&-
 k\prod_t^s (\mat{I}+\mat{M}(u)\,\dd u)\mat{R}(s)\mat{m}^{(k-1)}(s,t) \\
 &&- \sum_{m=2}^k \begin{pmatrix}
  k \\
  m
  \end{pmatrix}
  \prod_t^s (\mat{I}+\mat{M}(u)\,\dd u)\left( \mat{D}(s)\bullet \mat{B}^{\bullet m}(s)\right) \mat{m}^{(k-m)}(s,t) .
 \end{eqnarray*}
 Integrating the equation then gives
 \begin{eqnarray}
 \mat{m}^{(k)}(s,t)&=&k \int_s^t \mat{P}(s,x)\mat{R}(x)\mat{m}^{(k-1)}(x,t)\,\dd x \nonumber \\
 &&+ \sum_{m=2}^k \begin{pmatrix}
  k \\
  m
  \end{pmatrix} \int_s^t \mat{P}(s,x)\left( \mat{D}(x)\bullet \mat{B}^{\bullet m}(x)  \right) \mat{m}^{(k-m)}(x,t)\,\dd x   \nonumber \\
  &\stackrel{\eqref{def-symbol:C}}{=}& k \int_s^t \mat{P}(s,x)\mat{R}(x)\mat{m}^{(k-1)}(x,t)\,\dd x \nonumber \\
 &&+ \sum_{m=2}^k \begin{pmatrix}
  k \\
  m
  \end{pmatrix} \int_s^t \mat{P}(s,x) 
  \mat{C}^{(m)}(x) \mat{m}^{(k-m)}(x,t)\,\dd x .
  \label{eq:moment-integrals}
 \end{eqnarray}
 % For $r(x)=0$, the matrix
 %  $\mat{F}^{(k)}(u)$ amounts to 
 % \begin{eqnarray*}
 % \mat{F}^{(k)}(u)&=&
 % \arraycolsep=3.8pt\def\arraystretch{2.2}
 % \left(\begin{array}{ccccccc}
 % \mat{M}(u) & {k\choose 1}\mat{R}(u) & {k\choose 2}\mat{C}^{(2)}(u) &  \cdots &  {k\choose k-1}\mat{C}^{(k-1)}(u) & \mat{C}^{(k)}(u) \\
 % \mat{0} & \mat{M}(u) & {k-1 \choose 1}\mat{R}(u) &\cdots & {k-1\choose k-2}\mat{C}^{(k-2)}(u) &\mat{C}^{(k-1)}(u) \\
 % \vdots & \vdots & \vdots &  \vdots \vdots \vdots & \vdots &\vdots  \\
 % \mat{0} &\mat{0} &\mat{0} & \cdots & \mat{M}(u) &\mat{R}(u) \\
 % \mat{0} &\mat{0} &\mat{0} &\cdots  &\mat{0} & \mat{M} (u)
 % \end{array}\right) .
 % \end{eqnarray*}
 Now we employ an induction argument to prove the identity of the product integral of the above matrix indeed equals $\mat{H}^{(k)}(s,t)$. 
 % \begin{eqnarray*}
 % \mat{H}^{(k)}(s,t)&=&
 % \arraycolsep=3.8pt\def\arraystretch{2.2}
 % \left(\begin{array}{ccccccc}
 % \mat{P} & {k\choose 1}\mat{m}^{(1)} & {k\choose 2}\mat{m}^{(2)} &  {k\choose 3}\mat{m}^{(3)} & \cdots &  {k\choose k-1}\mat{m}^{(k-1)} & \mat{m}^{(k)} \\
 % \mat{0} & \mat{P} & {k-1 \choose 1}\mat{m}^{(1)} & {k-1\choose 2}\mat{m}^{(2)} & \cdots & {k-1\choose k-2}\mat{m}^{(k-2)} &\mat{m}^{(k-1)} \\
 % \mat{0} &\mat{0} &\mat{P} & {k-2\choose 1}\mat{m}^{(1)} & \cdots &  {k-2\choose k-3}\mat{m}^{(k-3)}  & \mat{m}^{(k-2)} \\
 % \vdots & \vdots & \vdots & \vdots & \vdots \vdots \vdots & \vdots &\vdots  \\
 % \mat{0} &\mat{0} &\mat{0} &\mat{0} & \cdots & \mat{P} &\mat{m}^{(1)} \\
 % \mat{0} &\mat{0} &\mat{0} &\mat{0} & \cdots  &\mat{0} & \mat{P} 
 % \end{array}\right),
 % \end{eqnarray*}
 For $k=1$ the results amounts to 
\[   \prod_s^t \left(\mat{I}+
\begin{pmatrix}
\mat{M}(u) & \mat{R}(u) \\
\mat{0} & \mat{M}(u)
\end{pmatrix}
\,\dd u\right)  =  \begin{pmatrix}
\mat{P}(s,t) & \mat{m}^{(1)}(s,t) \\
\mat{0} & \mat{P}(s,t)
\end{pmatrix} \]
which indeed holds true since Lemma \ref{lemma:van-loan} implies that
\begin{equation*}
 \mat{m}^{(1)}(s,t) = \int_s^t \mat{P}(s,x)
\mat{R}(x)  \mat{P}(x,t) \,\dd x , \label{eq:first-moment-int}
\end{equation*}
which has been previously established in Lemma \ref{lemma:rewards}. 
Assume that the results hold true for dimension $k-1$. 
Partition the matrix $\mat{F}^{(k)}(u)$ as
\[ \mat{F}^{(k)}(u) = 
\begin{pmatrix}
\mat{M}(u) & \mat{x}^{(k)}(u) \\
\mat{0} & \mat{F}^{(k-1)}(u)
\end{pmatrix} ,
  \]
 where
 \[ \mat{x}^{(k)}(u) = \left(  {k\choose 1}\mat{R},\  {k\choose 2}\mat{C}^{(2)},\   {k\choose 3}\mat{C}^{(3)},  \cdots ,\   {k\choose k-1}\mat{C}^{(k-1)},\  {k \choose k}\mat{C}^{(k)}  \right) . \] 
 Then use Lemma \ref{lemma:van-loan}, the induction hypothesis and \eqref{eq:moment-integrals} to verify the correct form of the first block row. 
 \end{proof}

\begin{remark}
The central moments 
\[  \Exp \Bigl[U(s,t)-\Exp \bigl[U(s,t)\,\,|Z(s)=i\bigr]\Bigr]^k\,\Big|\, Z(s)=i \Bigr]\]
can be obtained by Theorem \ref{th:main-moments} by a simple reparametrisation, replacing 
$\vect{b}(t)=(b^1(t),...,b^p(t))$ by 
\[  \vect{b}(t)-\frac{\Exp \bigl[U(s,t)\,\big|\,Z(s)=i\bigr]}{t-s} \vect{e},  \] where the expected values 
$\Exp \bigl[U(s,t)\,\big|\,Z(s)=i\bigr]$ are then first calculated by Theorem ~\ref{th:main-moments} in the usual way with $k=1$.
\end{remark}

\begin{example}[Markovian Arrival Process]
Let $\{ Z(t)\}_{t\geq 0}$ be a time--homogeneous Markov jump process with state--space $E=\{1,2,...,p\}$ and intensity matrix
\[ \mat{\Lambda}=\mat{C}+\mat{D} , \]
where $\mat{C}$ is a sub--intensity matrix and $\mat{D}$ a non--negative matrix. A Markovian Arrival Process $N$ is a point process which is constructed in the following way. Upon transitions from $i$ to $j$ of $Z(t)$, an arrival of $N$ is produced with probability $d_{ij}/(c_{ij}+d_{ij})$, and during during sojourns in state $i$, there are Poisson arrivals at rate $d_{ii}$. Let $N(0,t)$ denote the number of arrivals in the MAP during $[0,t]$. Then we may calculate the moments of $N(0,t)$ by the use of Theorem \ref{th:main-moments}. We identify $\mat{M}(s)=\mat{C}+\mat{D}$, $\mat{B}(t) = \mat{E}$ (the matrix of ones), $\vect{b}(s)=\vect{0}$,  $\mat{R}(s)=\mat{D}$, $\mat{C}^{(k)}(s)=\mat{D}$ and
\[  \mat{F}^{(k)}=\mat{F}^{(k)}(x)=
\arraycolsep=3.8pt\def\arraystretch{2.2}
\left(\begin{array}{ccccccc}
\mat{C}+\mat{D} & {k\choose 1}\mat{D} & {k\choose 2}\mat{D}&  \cdots &  {k\choose k-1}\mat{D} & \mat{D} \\
\mat{0} & \mat{C}+\mat{D} & {k-1 \choose 1}\mat{D} & \cdots & {k-1\choose k-2}\mat{D} &\mat{D} \\
%\mat{0} &\mat{0} &\mat{M}-(k-2)r(x)\mat{I} & {k-2\choose 1}\mat{R} & \cdots &  {k-2\choose k-3}\mat{C}^{(k-3)}  & \mat{C}^{(k-2)} \\
\vdots & \vdots & \vdots &  \vdots \vdots \vdots & \vdots &\vdots  \\
\mat{0} &\mat{0} &\mat{0} & \cdots & \mat{C}+\mat{D} &\mat{D} \\
\mat{0} &\mat{0} &\mat{0} & \cdots  &\mat{0} & \mat{C}+\mat{D}
\end{array}\right) .  \]
Then the moments 
\[  \mat{m}^{(k)}(s,t)=\bigg\{ \Exp \left. \left( 1\{ Z(t)=j \} N(s,t)^k \right|Z(s)=i \right) \bigg\}_{i,j\in E} \]
are obtained through the matrix--exponential of $\mat{F}^{(k)}$ as
\begin{eqnarray*}
\arraycolsep=3.8pt\def\arraystretch{2.2}
\left(\begin{array}{ccccccc}
\mat{P}(s,t) & {k\choose 1}\mat{m}^{(1)}(s,t) & {k\choose 2}\mat{m}^{(2)}(s,t) &  \cdots &  {k\choose k-1}\mat{m}^{(k-1)}(s,t) & \mat{m}^{(k)}(s,t) \\
\mat{0} & \mat{P}(s,t) & {k-1 \choose 1}\mat{m}^{(1)}(s,t) & \cdots & {k-1\choose k-2}\mat{m}^{(k-2)}(s,t) &\mat{m}^{(k-1)}(s,t)\\
\vdots & \vdots & \vdots & \vdots \vdots \vdots & \vdots &\vdots  \\
\mat{0} &\mat{0} &\mat{0} & \cdots & \mat{P}(s,t) &\mat{m}^{(1)}(s,t) \\
\mat{0} &\mat{0} &\mat{0} & \cdots  &\mat{0} & \mat{P}(s,t)
\end{array}\right)  
= \exp \left(  \mat{F}^{(k)} (t-s) \right) .
\end{eqnarray*}
In particular, the conditional moments 
\[  
\begin{pmatrix}
\Exp \left. \left( N(s,t)^k \right|Z(s)=1 \right) \\
 \Exp \left. \left( N(s,t)^k \right|Z(s)=2 \right) \\
 \vdots \\
 \Exp \left. \left( N(s,t)^k \right|Z(s)=p \right)
\end{pmatrix} = \exp \left(  \mat{F}^{(k)} (t-s) \right)\vect{e} .
  \]
The factorial moments
\[  \mat{fm}^{(k)}(s,t) = \bigg\{  \Exp \left. \left( 1\{ Z(t)=j \} N(s,t)(N(s,t)-1)\cdots (N(s,t)-k+1) \right|Z(s)=1 \right)  \bigg\}_{i,j\in E} \]
are similarly obtained by the formula 
\begin{eqnarray*}
\lefteqn{\arraycolsep=3.8pt\def\arraystretch{2.2}
\left(\begin{array}{ccccccc}
\mat{P}(s,t) & \mat{fm}^{(1)}(s,t) & \mat{fm}^{(2)}(s,t) &  \cdots &  \mat{fm}^{(k-1)}(s,t) & \mat{fm}^{(k)}(s,t) \\
\mat{0} & \mat{P}(s,t) & \mat{fm}^{(1)}(s,t) & \cdots & \mat{fm}^{(k-2)}(s,t) &\mat{fm}^{(k-1)}(s,t)\\
\vdots & \vdots & \vdots & \vdots \vdots \vdots & \vdots &\vdots  \\
\mat{0} &\mat{0} &\mat{0} & \cdots & \mat{P}(s,t) &\mat{fm}^{(1)}(s,t) \\
\mat{0} &\mat{0} &\mat{0} & \cdots  &\mat{0} & \mat{P}(s,t)
\end{array}\right) }~~\\ 
&=& \exp \left(   
\begin{pmatrix}
\mat{C}+\mat{D} & \mat{D} & \mat{0} & \mat{0} & \cdots & \mat{0} & \mat{0} \\
\mat{0} & \mat{C}+\mat{D} & \mat{D} & \mat{0}  & \cdots & \mat{0} & \mat{0} \\
\vdots & \vdots & \vdots & \vdots & \vdots \vdots \vdots & \vdots &\vdots  \\
\mat{0} &\mat{0} &\mat{0} & \mat{0} &\cdots & \mat{C}+\mat{D} &\mat{D} \\
\mat{0} &\mat{0} &\mat{0} & \mat{0} &\cdots & \mat{0} &\mat{C}+\mat{D}  \\
\end{pmatrix} (t-s)
\right) .
\end{eqnarray*}
Integral representations of the moments and factorial moments in a MAP has been considered in \cite{bo-Uffe-2007}.\qed
\end{example}

Next we turn to the case of discounted rewards (future payments). In principle we may calculate the moments of the future discounted payments by applying Theorem \ref{th:main-moments} to the discounted rewards
\[  \e^{-\int_s^ur(x)\,\dd s}b^i(u) \ \ \mbox{and}\ \ \e^{-\int_s^ur(x)\,\dd s}b^{ij}(u) . \]
We may however obtain an explicit matrix representation which also involves the interest rate $r(x)$ in a more convenient way and which is closer to standard intuition in life insurance (like e.g. Hattenforffs theorem). We define
$\mat{F}_U^{(k)}(x)
$ as the matrix
\begin{equation}
\arraycolsep=3.8pt\def\arraystretch{2.2}
\left(\begin{array}{ccccccc}
\mat{M}(x)-kr(x)\mat{I} & {k\choose 1}\mat{R}(x) & {k\choose 2}\mat{C}^{(2)}(x) &  \cdots &  {k\choose k-1}\mat{C}^{(k-1)}(x) & \mat{C}^{(k)}(x) \\
\mat{0} & \mat{M}(x)-(k-1)r(x)\mat{I} & {k-1 \choose 1}\mat{R}(x) & \cdots & {k-1\choose k-2}\mat{C}^{(k-2)}(x) &\mat{C}^{(k-1)}(x) \\
%\mat{0} &\mat{0} &\mat{M}-(k-2)r(x)\mat{I} & {k-2\choose 1}\mat{R} & \cdots &  {k-2\choose k-3}\mat{C}^{(k-3)}  & \mat{C}^{(k-2)} \\
\vdots & \vdots & \vdots &  \vdots \vdots \vdots & \vdots &\vdots  \\
\mat{0} &\mat{0} &\mat{0} & \cdots & \mat{M}(x)-r(x)\mat{I} &\mat{R}(x) \\
\mat{0} &\mat{0} &\mat{0} & \cdots  &\mat{0} & \mat{M}(x) 
\end{array}\right) \label{eq:F-gen-res}
\end{equation}
and let 
\[  \mat{G}^{(k)}(s,t)= \prod_s^t (\mat{I} + \mat{F}_U^{(k)}(x)\,\dd x) .  \]
The matrix $\mat{F}_U^{(k)}(x)$ is a $(k+1)p\times (k+1)p$ block--partitioned matrix with blocks of
sizes $p\times p$. Thus $\mat{G}^{(k)}(s,t)$ is also a $(k+1)p\times (k+1)p$ matrix, and we define
a similar block partitioning as for $\mat{F}^{(k)}(x)$, letting $\mat{G}_{ij}^{(k)}(s,t)$ denote
the $ij$'th block which corresponds to the $ij$'th block of $\mat{F}_U^{(k)}(x)$. Then we have the following main result.
\begin{theorem}\label{th:main-reserve}
For $j=1,...,k$ we have that
\[   \mat{V}^{(j)}(s,t) = \mat{G}_{k+1-j,k+1}^{(k)}(s,t) , \]
whereas 
\[  \mat{P}(s,t) =  \mat{G}_{k+1,k+1} . \]
\end{theorem}
The theorem states that the right block--column of $\mat{G}^{(k)}(s,t)$ contains the moments 
$ \mat{V}^{(j)}(s,t)$, starting with the highest moment in the upper right corner and finishing with the
transition matrix in the lower right corner (which may be considered as the zeroth moment). Symbolically, 
\begin{equation}
  \prod_s^t (\mat{I} + \mat{F}^{(k)}(x)\,\dd x) 
=
\begin{pmatrix}
* & * & * & * & \cdots * & \mat{V}^{(k)}(s,t) \\
* & * & * & * & \cdots * & \mat{V}^{(k-1)}(s,t) \\
* & * & * & * & \cdots * & \mat{V}^{(k-2)}(s,t) \\
\vdots & \vdots & \vdots &\vdots & \vdots\vdots\vdots & \vdots & \vdots \\
* & * & * & * & \cdots * & \mat{V}^{(1)}(s,t) \\
* & * & * & * & \cdots * & \mat{P}(s,t) \\
\end{pmatrix} . \label{eq:symbolically}
\end{equation}

The idea of the general proof is most easily explained through the following example, which proves the result of Theorem \ref{th:main-reserve} for the case $k=2$, which is the lowest non--trivial order.
\begin{example}[quadratic moment]\label{ex:2-moment}
First we consider the product integral
\[  \mat{G}(s,t) = \prod_s^t \left(  \mat{I}  + 
\begin{pmatrix}
\mat{A}_{11}(x) & \mat{A}_{12}(x) & \mat{A}_{13}(x) \\
\mat{0} & \mat{A}_{22}(x) & \mat{A}_{23}(x) \\
\mat{0} & \mat{0} & \mat{A}_{33}(x)  .
\end{pmatrix} \,\dd x
  \right)  = 
 \begin{pmatrix}
 \mat{G}_{11}(s,t) & \mat{G}_{12}(s,t) & \mat{G}_{13}(s,t) \\
 \mat{0} &  \mat{G}_{22}(s,t) & \mat{G}_{23}(s,t) \\
 \mat{0} & \mat{0} & \mat{G}_{33}(s,t)
 \end{pmatrix} .  \]
Employing Lemma \ref{lemma:van-loan} inductively, by first partioning the matrix as
\[ \left(\begin{array}{c|cc}
\mat{A}_{11}(x) & \mat{A}_{12}(x) & \mat{A}_{13}(x)  \\ \hline
\mat{0} & \mat{A}_{22}(x) & \mat{A}_{23}(x) \\
\mat{0} & \mat{0} & \mat{A}_{33}(x)  
\end{array}\right) . \]
we see that
\begin{eqnarray*}
\mat{G}_{11}(s,t)&=&\prod_s^t (\mat{I} + \mat{A}_{11}(x)\,\dd x) \\
\begin{pmatrix}
\mat{G}_{22}(s,t) & \mat{G}_{23}(s,t) \\
\mat{0} & \mat{G}_{33}(s,t)
\end{pmatrix} &=&
\prod_s^t \left(   
\mat{I} +
\begin{pmatrix}
\mat{A}_{22}(x) & \mat{A}_{23}(x) \\
\mat{0} & \mat{A}_{33}(x)
\end{pmatrix}\,\dd x
\right) \\
&=& \begin{pmatrix}
\prod_s^t (\mat{I} + \mat{A}_{22}(x)\,\dd x) & \int_s^t \prod_s^x (\mat{I} + \mat{A}_{22}(u)\,\dd u)
\mat{A}_{23}(x) \prod_x^t (\mat{I} + \mat{A}_{11}(u)\,\dd u)\,\dd x
  \\
\mat{0} & \prod_s^t (\mat{I} + \mat{A}_{33}(x)\,\dd x) 
\end{pmatrix}
\end{eqnarray*}
whereas
\begin{eqnarray*}
\left(\mat{G}_{12}(s,t), \mat{G}_{13}(s,t)\right)&=&
\int_s^t \prod_s^x (\mat{I}+\mat{A}_{11}(u)\,\dd u) \left( \mat{A}_{12}(x),\ \mat{A}_{13}(x)  \right) \prod_{x}^t \left( \mat{I} +  
\begin{pmatrix}
\mat{A}_{22}(u) & \mat{A}_{23}(u) \\
\mat{0} & \mat{A}_{33}(u)
\end{pmatrix}\,\dd u
 \right) \,\dd x
\end{eqnarray*}
so that
\begin{eqnarray*}
\mat{G}_{12}(s)&=&\int_s^t \prod_s^x (\mat{I}+\mat{A}_{11}(u)\,\dd u) \mat{A}_{12}(x)\prod_x^t (\mat{I}+\mat{A}_{22}(u)\,\dd u)\,\dd x
\end{eqnarray*}
and
\begin{eqnarray*}
\mat{G}_{13}(s)&=&\int_s^t \prod_s^x (\mat{I}+\mat{A}_{11}(u)\,\dd u) \mat{A}_{13}(x)\prod_x^t (\mat{I}+\mat{A}_{33}(u)\,\dd u)\,\dd x \\
&&+ \int_s^t \prod_s^x (\mat{I}+\mat{A}_{11}(u)\,\dd u)\mat{A}_{12}(x) 
\int_x^t \prod_x^y (\mat{I}+\mat{A}_{22}(u)\,\dd u)\mat{A}_{23}(y) \prod_y^t (\mat{I}+\mat{A}_{33}(u)\,\dd u)dy .
\end{eqnarray*}
Now assume that we are concerned with the discounted prices. Then at any time $x\in [s,t]$, we discount the price by 
\[   \exp \left(  -\int_s^x r(u)\,\dd u  \right) .  \]
In the above expression for $\mat{G}_{13}(s,t)$, 
\[  \mat{A}_{13}(x) = \mat{C}^{(2)}(x) = \mat{D}(x)\bullet \mat{B}(x)\bullet \mat{B}(x)  \]
while 
\[ \mat{A}_{12}(x)=\mat{A}_{23}(x)=\mat{R}(x)=\mat{D}(x)\bullet \mat{B}(x)+\mat{\Delta}(\vect{b}(x)).  \]
In the expression for $\mat{G}_{13}(s,t)$, $\mat{A}_{13}(x)$ produces a discount of 
\[ \exp \left( -2\int_s^x r(u)\,\dd u \right) , \]
$\mat{A}_{12}(x)$ a discount of 
\[ \exp \left( -\int_s^x r(u)\,\dd u \right)  \]
while $\mat{A}_{23}(y)$ produces a discount of
\[  \exp \left( -\int_s^y r(u)\,\dd u \right) = \exp \left( -\int_s^x r(u)\,\dd u \right)\exp \left( -\int_x^y r(u)\,\dd u \right) . \]
Thus we may write
\begin{eqnarray*}
\mat{G}_{13}(s,t)&=& \int_s^t \prod_s^x (\mat{I}+[\mat{A}_{11}(u)-2r(u)\mat{I}]\,\dd u) \mat{A}_{13}(x)\prod_x^t (\mat{I}+\mat{A}_{33}(u)\,\dd u)\,\dd x \\
&&\hspace{-2cm}+ \int_s^t \prod_s^x (\mat{I}+[\mat{A}_{11}(u)-2r(u)\mat{I}]\,\dd u)\mat{A}_{12}(x) 
\int_x^t \prod_x^y (\mat{I}+[\mat{A}_{22}(u)-r(u)\mat{I}]\,\dd u)\mat{A}_{23}(y) \prod_y^t (\mat{I}+\mat{A}_{33}(u)\,\dd u)dy .
\end{eqnarray*}
Let
\[ \mat{H}^{(2)}(x) =
\begin{pmatrix}
\mat{M}(x)-2r(x)\mat{I} & 2\mat{R}(x) & \mat{C}^{(2)}(x) \\
\mat{0} & \mat{M}(x)-r(x)\mat{I} & \mat{R}(x) \\
\mat{0} & \mat{0} & \mat{M}(x)
\end{pmatrix} ,
 \]
 and
 \[ \mat{V}^{(2)}(s,t)=\prod_s^t (\mat{I}+\mat{H}^{(2)}(x)\,\dd x) = 
\begin{pmatrix}
 \mat{V}^{(2)}_{11}(s,t) & \mat{V}^{(2)}_{12}(s,t) & \mat{V}^{(2)}_{13}(s,t) \\
 \mat{0} & \mat{V}^{(2)}_{22}(s,t) & \mat{V}^{(2)}_{23}(s,t) \\
 \mat{0} &\mat{0} & \mat{V}^{(2)}_{33}(s,t)
\end{pmatrix} .
  \]
Then 
\begin{eqnarray*}
\mat{V}^{(2)}_{33}(s,t)&=& \mat{P}(s,t) \\
\mat{V}^{(2)}_{23}(s,t)&=& \left\{ \Exp\left. \left(1\{ Z(t)=j \} \mat{U}(s,t) \right| Z(s)=i \right)\right\}_{i,j} \\
\mat{V}^{(2)}_{13}(s,t)&=&\left\{ \Exp\left. \left(1\{ Z(t)=j \} \mat{U}^2(s,t) \right| Z(s)=i \right)\right\}_{i,j} .
\end{eqnarray*}
In particular,
\begin{eqnarray*}
\Exp (U(s,t)| Z(s)=i)&=&\vect{e}_i^\prime \mat{V}^{(2)}_{23}(s,t)\vect{e} \\
\Exp (U^2(s,t)| Z(s)=i)&=&\vect{e}_i^\prime \mat{V}^{(2)}_{13}(s,t)\vect{e}
\end{eqnarray*}
\qed

% Then, by \eqref{eq:moment-integrals}, 
% \begin{eqnarray*}
% \mat{m}^{(2)}(s,t)&=&2 \int_s^t \mat{P}(s,x)\mat{R}(x)\mat{m}^{(1)}(x,t)dx \nonumber +  \int_s^t \mat{P}(s,x)\left( \mat{M}(x)\bullet \mat{B}^{\bullet 2}(x)  \right) \mat{m}^{0}(x,t)dx \\
% &=&2 \int_s^t \mat{P}(s,x)\mat{R}(x)\mat{m}^{(1)}(x,t)dx \nonumber +  \int_s^t \mat{P}(s,x)\left( \mat{M}(x)\bullet \mat{B}^{\bullet 2}(x)  \right)\mat{P}(x,t)dx \\
% &=&2 \int_s^t \int_x^t \mat{P}(s,x)\mat{R}(x)\mat{P}(x,u)\mat{R}(u)\mat{P}(u,t)du\ dx \\
% &&+ \int_s^t \mat{P}(s,x)\left( \mat{M}(x)\bullet \mat{B}^{\bullet 2}(x)  \right)\mat{P}(x,t)dx .
% \end{eqnarray*}
% Let us now incorporate the discounting factors. 
\end{example}
\noindent We now turn to the general proof. 

\begin{proof}[Proof of Theorem \ref{th:main-reserve}]
We apply Theorem \ref{th:main-moments} to the discounted prices
\[  \e^{-\int_s^ur(x)\,\dd s}b^i(u) \ \ \mbox{and}\ \ \e^{-\int_s^ur(x)\,\dd s}b^{ij}(u) . \]
This will indeed provide us with the correct result (for fixed $s$), and as in 
Example \ref{ex:2-moment} we redistribute the discounted terms into the block diagonal matrices.
For simplicity of identification of the individual blocks of the matrix, we write $\mat{F}^{(k)}_{U}(u)$ on a black partitioned way as
\[   \mat{F}_U^{(k)}(u) =
\begin{pmatrix}
\mat{A}_{11}(u) & \mat{A}_{12}(u) & \cdots &\mat{A}_{1,k+1}(u) \\
\mat{0} & \mat{A}_{22}(u) & \cdots &\mat{A}_{2,k+1}(u) \\
\vdots & \vdots & \vdots\vdots\vdots & \vdots \\
\mat{0} & \mat{0} & \cdots &\mat{A}_{k+1,k+1}(u) \\
\end{pmatrix} 
 \]
 and
 \[
\prod_s^t (\mat{I}+\mat{F}_U^{(k)}(u)\,\dd u) = 
\begin{pmatrix}
\mat{B}_{11}(s,t) & \mat{B}_{12}(s,t) & \cdots &\mat{B}_{1,k+1}(s,t) \\
\mat{0} & \mat{B}_{22}(s,t) & \cdots &\mat{B}_{2,k+1}(s,t) \\
\vdots & \vdots & \vdots\vdots\vdots & \vdots \\
\mat{0} & \mat{0} & \cdots &\mat{B}_{k+1,k+1}(s,t) \\
\end{pmatrix} .
\]
 For example, $\mat{A}_{ii}(u)=\mat{M}(u)$ and $\mat{A}_{i,i+1}(u)=(k-i+1)\mat{R}(u)$. The matrix
 $\mat{A}_{i,i+m}(x)$ is then scaled by 
 \[ \exp (-m\int_s^x r(u)\,\dd u) . \]
For $k=1$ it is clear that 
\[ \mat{B}_{k,k+1}(s,t) = \int_s^t \prod_{s}^x (\mat{I}+\mat{A}_{kk}(u)\,\dd u)\mat{A}_{k,k+1}(x)
\prod_x^t (\mat{I}+\mat{A}_{k+1,k+1}(u)\,\dd s)\ \,\dd x  \]
has scaling factor $\exp (-\int_s^x r(u)\,\dd u)$, while in Example \ref{ex:2-moment} we saw that 
$ \mat{B}_{k-1,k+1}(s,t)$ has scaling factor  $\exp (-2\int_s^x r(u)\,\dd u)$. Now assume (induction) that all 
$\mat{B}_{j,k+1}(s,t)$, $j=i+1,...,k$ have scaling factors $\exp (-(k-j+1)\int_s^x r(u)\,\dd u)$. From the recursion \eqref{eq:recursion},
\[ \mat{B}_{i,k+1}(s,t) = \sum_{j=i+1}^{k+1} \int_s^t \prod_s^x (\mat{I}+\mat{A}_{ii}(u)\,\dd u)\mat{A}_{ij}(x)\mat{B}_{j,k+1}(x,t)\,\dd x ,\]
we see that $\mat{A}_{ij}(x)$ will produce a scaling factor $\exp (-(j-i)\int_s^x r(u)\,\dd u)$,
while $\mat{B}_{j,k+1}(x,t)$ can be written as another integral over $x$ to $t$ with integration variable $y$, say, which will then have scaling factors (induction) of size $\exp (-(k-j+1)\int_s^y r(u)\,\dd u)$. Now write
\[  \exp \left(-(k-j+1)\int_s^y r(u)\,\dd u\right) = \exp \left(-(k-j+1)\int_s^x r(u)\,\dd u\right)\exp \left(-(k-j+1)\int_x^y r(u)\,\dd u\right) \]
and pull out the factor $\exp (-(k-j+1)\int_s^x r(u)\,\dd u)$ to get a scaling factor of 
\[   \exp \left(-(j-i)\int_s^x r(u)\,\dd u\right)\exp \left(-(k-j+1)\int_s^x r(u)\,\dd u\right) = \exp \left(-(k-i+1)\int_s^x r(u)\,\dd u\right)  . \]
This scaling factor can then be put together with $\prod_s^x (\mat{I}+\mat{A}_{ii}(u)\,\dd u)$, which in turn is the 
$(i,i)$ block level entrance which is simply $\prod_s^x (\mat{I}+\mat{M}(u)\,\dd u)$.
\end{proof}
We may then obtain a slightly generalised version of Hattendorff's theorem.
  \begin{theorem}
\[  \frac{\partial}{\partial s}\mat{V}^{(k)}(s,t) = \left( kr(s)\mat{I} -\mat{M}(s)  \right) \mat{V}^{(k)}(s,t) - k\mat{R}(s)\mat{V}^{(k-1)}(s,t) - \sum_{i=2}^k {k \choose i} \mat{C}^{(i)}(s)\mat{V}^{(k-i)}(s,t) ,  \]
with terminal condition $\mat{V}^{(k)}(t,t)=\mat{0}$. 
\end{theorem}
\begin{proof}
Follows from differentiation of $\prod_s^t (\mat{I} + \mat{F}^{(k)}(x)\,\dd x)$ with respect to $s$, obtaining a Kolmogorov type of differential equation, with $\mat{F}^{(k)}(x)$ given by 
\eqref{eq:F-gen-res}, and comparing to \eqref{eq:symbolically}. We only need the first block row of $\mat{F}^{(k)}(x)$ and the last block column of  $\prod_s^t (\mat{I} + \mat{F}^{(k)}(x)\,\dd x)$.
\end{proof}
This theorem reduces to the state--wise standard Hattendorff theorem for $k$th order moments, which is achieved by post--multiplying the differential equation by the vector $\vect{e}=(1,1,...,1)^\prime$.

\section{Gram-Charlier expansions of the full distribution}
The c.d.f.\ or density of $X=U(s,T)$ can of course be evaluated by Laplace transform inversion 
 from Theorem~\ref{Th:4.5a}, say
implementing via the Euler or Post--Widder methods (see \cite{Abate1995}). However, the procedure is somewhat tedious and given the availability of all moments, an
attractive alternative is Gram-Charlier expansions via orthogonal polynomials.

The method can briefly be summarized as follows. Consider a reference density $f_0(x)$ having all moments
$\int x^kf_0(x)\,\dd x$ well-defined and finite, and a target density $f(x)$
for which all moments $\Exp X^k=$ $\int x^kf(x)\,\dd x$ can be computed. Consider $L_2(f_0)$ with inner product
$\langle g,h\rangle=$ $\int g(x)h(x)f_0(x)\,\dd x$ and let $p_0(x), p_1(x),\ldots$ be a set of orthonormal
polynomials, i.e.\ $\langle p_n,p_m\rangle=$ $\delta_{nm}$. If this set is complete in $L_2(f_0)$ and
\begin{equation}\label{5.4a}
f/f_0\in L_2(f_0)\,,\quad\text{i.e.}\ \ \int \frac{f^2(x)}{f_0(x)}\,\dd x\,<\,\infty\quad\text{or
equivalently }f^2/f_0\in L_1(\text{Leb})\\,,\end{equation}
we can then expand $f/f_0$ in the $p_n$ to get
\begin{equation}\label{5.4b} f(x)\ =\ f_0(x)\Bigl\{1+\sum_{n=3}^\infty c_np_n(x)\Bigr\}\quad\text{where}\ \ 
c_n\,=\, \langle f/f_0,p_n\rangle \,=\,\Exp p_n(X)\,.\end{equation}
If the emphasis is on the c.d.f.\ $F$ or the quantiles, simply integrate this to get an expansion of $F(x)$.

For fast convergence of the series \eqref{5.4b}, $f_0$ should be chosen as much alike $f$ as possible.
The most popular choice   is the normal density with the same mean $\mu$ and variance $\sigma^2$ as
$f$, in which case $c_1=c_2=0$ (one has always $c_0= 1$). This implies $p_n(x)=$
$H_n\bigl((x-\mu)/\sigma\bigr)/\sqrt{n!}$ for $n\ge 1$ where $H_n$ is the $n$th (probabilistic)
Hermite polynomial defined by $(\dd^n/\dd x)^n\e^{-x^2/2}$ $=(-1)^nH_n(x)\e^{-x^2/2}$. In particular, 
with \[d_n=\frac{1}{n!}\int_{-\infty}^\infty H_n((x-\mu)/\sigma\bigr)f(x)\,\dd x\]
we have
\begin{align}\label{5.4cf} f(x)\ &=\ \frac{1}{\sigma\sqrt{2\pi}}\e^{(x-\mu)^2/2\sigma^2}
\Bigl\{1+\sum_{n=3}^\infty d_nH_{n}\bigl((x-\mu)/\sigma\bigr)\Bigr\}\,,\\ \label{5.4c}
F(x)\ &=\ \Phi\bigl((x-\mu)/\sigma\bigr)- \frac{1}{\sigma\sqrt{2\pi}}\e^{(x-\mu)^2/2\sigma^2}\sum_{n=3}^\infty d_nH_{n-1}\bigl((x-\mu)/\sigma\bigr)\,.
\end{align}
The conditions for \eqref{5.4c} to be a valid expansion are in fact just 
\begin{equation}\label{4.5a}
f/f_0^{1/2}\in L_1(\text{Leb})\,,\quad\text{i.e.}\ \ \int \e^{(x-\mu)^2/4\sigma^2}f(x)\,\dd x\,<\,\infty\,,\end{equation}
cf.\ \cite[p.\,223]{cramer1946}.
Truncated versions of \eqref{5.4b} or \eqref{5.4c}  go under the name of Edgeworth expansions;
the examples with the whole series not  converging  simply arise when conditions \eqref{5.4a} or
\eqref{4.5a} is violated (whereas completeness holds when $f_0$ is normal).
See, e.g., \cite{Szego39}, \cite[p.\,133,\,222ff.]{cramer1946} and \cite{Barndorff2013asymptotic} for more detail. Actuarial applications of the method can be found, e.g., in \cite{Bowers1966}, \cite{ALBRECHER2001345},
\cite{GOFFARD2016499}, \cite{Goffard2017}, \cite{Goffard2019a}, \cite{Goffard2019b}.

\subsection{The insurance implementation}
When implementing the method in the insurance context $X=U(s,T)$, 
a difficulty is that absolute continuity typically fails.
More precisely, the target distribution $F$ will be a mixture 
of atoms at $a_1,a_2,\ldots$, with probability
$q_i$ for $a_i$, and a part having a  density. 
One then has to take $f(x)$ as the density of the absolutely continuous part,
\[ f(x)\ =\ \Prob\bigl(X\in \dd x\,\big|\,X\ne a_1, X\ne a_2, \ldots\bigr)\]
Most often, there is only one atom with $q_1$ easily computable.
Examples of atoms:\\
1) the initial state $i$ is held throughout $(s,T]$, occuring w.p.\ $q_i=\int_s^T \e^{\int_s^T\mu_ii}$, so that
$a_1=$ $\int_s^T \e^{-r(u-s)} b^i(u)\,\dd  u$.\\
2) No discounting and equal lump sum payments, $b^{ij}(t)\equiv b^{ij}$, $r=0$. 
Then $U(s,T)$ is a linear combination of the $b^{ij}$.\\
These are more or less the only natural ways to get atoms that occur to us,
but see Remark~\ref{Rem:24.4a} below. \\
For simplicity of notation, we assume there is only one atom and write $a=a_1$, $q=q_1$
(the modifications in the case of several atoms are trivial).

To implement the Gram-Charlier expansion of $f$, define
\[m_1\,=\,\int_{-\infty}^\infty xf(x)\,\dd x\,,\quad m_k\,=\,\int_{-\infty}^\infty (x-m_1)^jf(x)\,\dd x\\,, j=2,3,\ldots\]
Obviously,
\begin{equation}\label{23.4a} \Exp\bigl[U(s,T)-m_1\bigr]^j\ =\ q(a-m_1)^j+(1-q)m_j, \ \ \  j\geq 2
\end{equation}
whereas for $j=1$, 
\begin{equation}\label{23.4b} \Exp\bigl[U(s,T)-m_1\bigr]\ =\ q(a-m_1) .
\end{equation}
%\newpage
The program for computing an $\alpha$-quantile $z_\alpha$ is then the following:
\begin{enumerate}
\item Compute $\Exp U(s,T)$ via Theorem 7.1 with $k=1$ 
and compute \[m_1\,=\,\frac{\Exp U(s,T) -qa}{1-q}\,.\]
\item  Choose $k>1$ and compute $\Exp\bigl[U(s,T)-m_1\bigr]^j$  for $j=2,\ldots,k$ via Theorem 7.1. To this end, replace the drift parameter
$a_i(t)$ by $a_i(t)-m_1/(T-s)$. Solve next \eqref{23.4a} for the $m_j$ to get
\[ m_j= \frac{\Exp\bigl[U(s,T)-m_1\bigr]^j-q(a-m_1)^j}{1-q} . \]
\item Take $f_0$ as the normal density with mean $m_1$ and variance $\sigma_f^2=m_2-m_1^2$. 
Write $H_n(x)=\sum_0^n a_{j;n}x^j$ and compute the
\[d_n\ =\ \frac{1}{n!}\sum_{j=0}^n\frac{a_{j;n}m_j}{\sigma^j}\,,\quad n=3,\ldots,k\]
\item Approximate the conditional density $f$ and the unconditional  c.d.f.\ $F$ by
\begin{align*} \widehat f_k(x)\ &=\ f_0(x)\Bigl\{1+\sum_{n=3}^k d_nH_{n}\bigl((x-\mu)/\sigma\bigr)\Bigr\}
\,,\\ \widehat F_k(x)\ &=\
q1_{x\ge a}+(1-q)\Biggl[\Phi\bigl((x-\mu)/\sigma\bigr)-f_0(x)\sum_{n=3}^\infty 
d_nH_{n-1}\bigl((x-\mu)/\sigma\bigr)\Biggr]\,.
\end{align*}
\item Solve $\widehat F_k(z_\alpha^k)=\alpha$ to get a candidate $z_\alpha^k$ for $z_\alpha$
\item Repeat from step 2) with a larger $k$ until $z_\alpha^k$ stabilizes.
\end{enumerate}

At the formal mathematical level, one needs to verify \eqref{5.4a}. This seems easier for feed forward systems, since then $U(s,T)$ has finite support and because a normal $f_0$ is bounded below on compact intervals,
it would suffice that $f$ is bounded. This may occur highly plausible, but 
Remark~\ref{Rem:24.4b} below shows that in fact \eqref{5.4a} may fail in complete generality.
The following, result seems, however, sufficient for all practical purposes.
We call a model with constant intensities $\mu_{ij}(t)\equiv \mu_{ij}$ \emph{feed-forward} if there are no loops,
i.e.\ no chain $i_0i_n\ldots i_N$ with $i_0=i_N$ and all $\mu_{i_{n-1}i_n}>0$.

\begin{theorem}\label{Th:29.4a} Assume as in Section x that all intensities and rewards are piecewise constant,
that the distribution of $U(s,T]$ has an absolutely continuous part with conditional density $f$ 
and that $f_0$  is normal$(\mu,\sigma^2)$.. 
Then the $L_2$ condition \eqref{5.4a} holds for $f$ if either \emph{(i)} the model  is feed-forward, 
\emph{(ii)}  $b^{ij}_k=0$ for all $i\ne j$ or, more generally,
\emph{(iii)} $|b^{ij}_k|>0$ implies that there is no path from $j$ to $i$, 
i.e.\ no chain $i_0i_n\ldots i_N$ with $i_0=j,i_N=i$ and all $\mu_{i_{n-1}i_n}>0$.
\end{theorem}

%\subsection{Issues relating to the $L_2$ condition  \eqref{5.4a}}
%{\tt Much may be moved to the Appendix ?}
\begin{proof} Let $\widetilde f(x)$ be the unconditional density and 
\[F(A\,|\,s,t;i,j)\ =\ \Prob\Bigl(\int_s^t \e^{-\int_s^u r(u)\,\dd u} \dd B(u),\, Z_t=j\,\Big|\,Z_s=i\Bigr)\ =\
F^{\mathrm{ac}}(A\,|\,s,t;i,j)+F^{\mathrm{at}}(A\,|\,s,t;i,j)\]
where $F^{\mathrm{ac}},F^{\mathrm{at}}$ are the absolutely, resp.\ atomic parts (due to the special structure,
there can be no singular but non-atomic part). Let $g(x\,|\,s,t;i,j)$ be the density of $F^{\mathrm{ac}}$
so that $\widetilde f(x) =$ $\sum_j g(x\,|\,0,T;i,j)$ when $Z_0=i$. Let
$\overline g(s,t;i,j)=$ $\sup_x g(x\,|\,s,t;i,j)$.
and assume  it shown that $\overline g(0,T-1;i,j)<\infty$, $\overline g(T-1,T;i,j)<\infty$. We then get
\begin{align*}g(0,T;i,j)\ &=\ \sum_k\int g(x-y\,|\,0,T-1;i,k)F(\dd y\,|\,T-1,T;k,j)\\
&\qquad\qquad + \sum_k\int g(x-y\,|\, T-1,T;k,j)F(\dd y\,|\,0,T-1;i,k)\\ &\le \
 \sum_k \overline g(0,T-1;i,k)\int F(\dd y\,|\,T-1,T;k,j)+ \sum_k\overline g(T-1,T;k,j)\int F(\dd y\,|\,0,T-1;i,k)
\end{align*}
so that  $\overline g(0,T;i,j)<\infty$.
Thus we may assume that $T=1$ and simply write $\mu_{ij,k}=\mu_{ij}$ etc.

The stated conditions imply in all three cases that $f$ has finite support. Indeed, the support is contained in
$[-A,A]$ where $A=p\max|b^i|+p(p-1)\max|b^{ij}|$ in cases (i) or (iii) and $A=p\max|b^i|$ in case (ii).
Since $f_0$ is bounded below on $[-A,A]$ for any $A$, it thus suffices to show that $f$ is bounded.

Assume first that all $b^{ij}=0$, , and let $x$ be fixed. 
Define $S_N\subset [0,1]^n$ as $S_N=$ $\{0<t_1<\cdots<t_N<1\}$ and $h_n=t_{n}-t_{n-1}$ for $0<n<N$, 
$h_0=t_1$, $h_N=1-t_N$.
A path $Z_0=i_0i_1\ldots i_N=Z_1$ contributes to $\widetilde f(x)$ only if $N>0$ and then by 
\begin{align*}\MoveEqLeft \int_{S_N}\Bigl[\prod_{n=1}^{N} \e^{\mu_{i_{n-1}i_{n-1}}h_n}\mu_{i_{n-1}i_{n}}\dd t_n\Bigr]\cdot \e^{\mu_{i_Ni_N}h_N}\cdot
{\bf 1}(a_1h_1+\cdots+a_Nh_N=x)\\ &\le 
\int_{S_N} \prod_{n=1}^{N}\mu_{i_{n-1}i_{n}}\dd t_n \ \le\ \overline\mu^n \cdot\text{Leb}(S_n)= 
\ =\ \frac{\overline\mu^N}{N!}
\end{align*}
where $\overline\mu^n=\max \mu_{ij}$. Thus the contribution from all paths of length $N$ is at most 
$p^N\mu^N{N!}$. Summing over $N$ gives the bound $\e^{p\overline\mu}$ for $f(x)$ which is independent of $x$.
\end{proof}

\begin{remark}\label{Rem:28.4a}\normalfont
Condition (iii) is not far from being necessary. Consider as an example 
the disability model with states 0:\,active, 1:\,disabled, 2:\,dead and recovery
in the interval time interval $[0,1]$ withthe
same intensity $\lambda>0$ for transitions from 0 to 1 as from 1 to 0 and (for simplicity)
mortality rate $0$ and discounting rate $r=0$. The benefits are a lump sum $b^{01}=1$ upon transition from 0 to 1
and the contributions a constant payment rate $b^0<0$ when active.

When $Z_0=0$, the total number $N$ of transitions in $[0,1]$ is Poisson$(\lambda)$, the total benefits $U_1(0,1)$
are $M=\lceil N/2\rceil$, and the total contributions $U_0(0,1)$ equal $|b^1|$ times the total time $T_0$ spent in state 0.
Thus $U(0,1)=$ $U_1(0,1)-U_0(0,1)$ $=M-U_0(0,1)$
Obviously, $U_0(0,1)$ is concentrated on the interval $[0,|b^1|]$ with a density $g(x)$ which is bounded away from
0, say $h(x)>c_1>0$. 

Assuming, again for simplicity, that $|b^1|>1$, the intervals $[m-[0,|b^1|],m]$ overlap and so for a given $x$, at least one of them contribute to $f(x)$. One candidate is the one with $m=\lceil x\rceil]$. This gives
\begin{align*}f(x)\ &\ge\ \Prob(M=\lceil x\rceil)h(\lceil x\rceil-x)\ \ge\ \Prob(N=2\lceil x\rceil)c_1\,.
\end{align*}
But using Stirling's approximation to estimate the Poisson probability, it follows after a little calculus that
$\Prob(N=2\lceil x\rceil)\ge$ $c\e^{-4x\log x}$ for all large $x$, say $x\ge x_0$, and some $c>0$
(in fact the 4 can be replaced by any $c_3>2$). This gives for a normal$(\mu,\sigma^2)$ $f_0$ that
\[\int\frac{f^2}{f_0}\ \ge\ \int_{x_0}^\infty\frac{f^2(x)}{f_0(x)}\,\dd x \ \ge\ 
c_4\int_{x_0}^\infty\exp\{(x-\mu)/\sigma^2-8x\log x\}\,\dd x\ =\ \infty\,.\]
That is, \eqref{5.4a} fails.

The obvious way out is of course to take $f_0$ with a heavier tail than the normal, say doubly exponential
(Laplacian) with density $\e^{-|x|}$ for $-\infty<x<\infty$. Given that this example and other 
cases where condition (iii) is violated do not seem very realistic, we have not exploited this further.
\end{remark}
% The proof is given in the next subsection together with an example showing some of the difficulties in verifying
% \eqref{5.4a} without piecewise constancy but also a possible approach in examples, 
% as well as a counterexample that things may go wrong.
\begin{example}\label{Ex:29.4a}\normalfont
\disabfig

Consider again the disability model with states 0:\,active, 1:\,disabled, 2:\,dead and no recovery.
As in BuchM, we assume that the payment stream has the form $b^0(t)=-b_-^0$ for $t\le S$
and $b^0(t)=b_+^0$ for $t>S$,\footnote{Here and in the following, a - subscript mimics `before $S$' and a + `after $S$'.}
$b^1(t)\equiv b^1$ for all $t$. For example $S$ could be the retirement age, say 65. Without recovery, the only non-zero transition rates
are the $\mu_{01}(t)$, $\mu_{02}(t)$, $\mu_{12}(t)$. With the values used in BuchM, these are 
bounded away from 0 and $\infty$ on $[0,T]$ for any $T$ (say 75 or 80), and this innocent assumption is all that matters for the following.

Define the stopping times
\begin{align*}
\tau_0^1\ &=\ \inf\{t>s:\, Z_t=1,\,Z_s=0\text{ for }s<t\}\\
\tau_1^2\ &=\ \inf\{t>\tau_0^1:\, Z_t=2\}\\
\tau_0^2\ &=\ \inf\{t>s:\, Z_t=2,\,Z_s=0\text{ for }s<t\}
\end{align*}
with the usual convention that the stopping time is $\infty$ if there is no $t$ meeting the requirement in the definition.
One then easily checks that the sets $F_0,\ldots,F_{1_+2_+}$ defined in Fig.~\ref{disabfig} defines a partition of the sample space.
Here $F_0$ (corresponding to zero transitions in $[s,T]$) contributes with an atom
at
\[a=-b_-^0(1-\e^{-rS })/r+b_+^0(\e^{- rS}-\e^{- rT})/r\text{\ \ with probability\ \ }
q=\exp\Bigl\{\int_0^T\mu_{00}(u)\,\dd u\Bigr\}\]
and the remaining 7 events with absolutely continuous parts, say with (defective) densities
$g_{2_-},\ldots,g_{1_+2_+}$. It is therefore sufficient to show that each of these is bounded.
In obvious notation, the contribution to $U(s,T)$ of the first 4 of these 7 events (corresponding to precisely one transition in $[s,T]$) are
\begin{gather*}
A_{2-}\ =\ \bigl[-b^0_-(1-\e^{-r\tau_0^2})/r\bigr]\cdot1_{\tau_0^2\le S}\,,\\
A_{2-}\ =\ \bigl[-b^0_-(1-\e^{-rS})/r+b_+^0(\e^{-rS}-\e^{-r\tau_0^2})/r\bigr]\cdot1_{S<\tau_0^2\le T}\,,\\
A_{1-}\ =\ \bigl[-b^0_-(1-\e^{-r\tau_0^1})/r+b^1(\e^{-rT}-\e^{-r\tau_0^1})/r\bigr]\cdot1_{\tau_0^1\le S}\,,\\
A_{1_+}\ =\ \bigl[-b^0_-(1-\e^{-rS})/r+b_+^0(\e^{-rS}-\e^{-r\tau_0^1})/r+b^1(\e^{-rT}-\e^{-r\tau_0^1})/r\bigr]\cdot1_{S<\tau_0^1\le T}\,.
\end{gather*}
The desired boundedness of $g_{2-},g_{2+},g_{1-},g_{1_+}$ therefore follows from
the following lemma, where $\tau$ may be improper ($\Prob(\tau=\infty)>0$) so that $\int h<1$:
\begin{lemma}\label{Lemma:24.4a} If a r.v.\ $\tau$ has a bounded density $h$ and the function $\varphi$ is monotone and differentiable with
$\varphi'$ bounded away from 0, then the density of $\varphi(\tau)$ is bounded as well.
\end{lemma}
\begin{proof} The assumptions imply the existence of $\psi=\varphi^{-1}$. Now just note that the density of $\varphi(\tau)$ is 
$h(\psi(x))/\varphi'(\psi(x))$ if $\varphi$ is increasing and $h(\psi(x))/\bigl|\varphi'(\psi(x))\bigr|$ if it is decreasing.
\end{proof}

The cases of $g_{1_-2_-},g_{1_-2_+},g_{1_+2_+}$. (corresponding to precisely two transitions in $(s,T]$) is slightly more intricate.
Consider first $g_{1_-2_-}$.
The contribution to $U(0,T)$ is here 
\[A_{1_-2_-}\ =\ -b_-^0\int_0^{\tau_0^1}\e^{-ru}\,\dd u\,+\, b^1\int_{\tau_0^1}^{\tau_1^2}\e^{-ru}\,\dd u\ =\ 
c_0+c_1\e^{-r\tau_0^1}-c_2\e^{-r\tau_1^2}\] 
Now the joint density $h(t_1,t_2)$ of
$(\tau_0^1,\tau_0^2)$ at a point $(t_1,t_2)$ with $0<t_1<t_2\le R$ is
\[\exp\Bigl\{\int_0^{t_1}\mu_{00}(u)\,\dd u\Bigr\}\mu_{01}(t_1)
\cdot\exp\Bigl\{\int_{t_1}^{t_2}\mu_{11}(v)\,\dd v\Bigr\}\mu_{12}(t_2)
\]
so that $h(t_1,t_2)$ is bounded. Consider now the transformation taking $(\tau_0^1,\tau_0^2)$ into 
$(\tau_0^1,A_{12_-})$. The inverse of the Jacobiant $J$ is
\[\begin{vmatrix} 1 & 0 \\ -rc_1\e^{-r\tau_0^1}&rc_2\e^{-r\tau_1^2}  \end{vmatrix}\ =\ rc_2\e^{-r\tau_1^2}\]
which is uniformly bounded away from 0 when $\tau_1^2\le S$. 
Therefore also the joint density $k(z_1,z_2)$ of
$(\tau_0^1,A_{12-})$ is bounded. 
Integrating $k(z_1,z_2)$w.r.t.\ $z_1$ 
over the finite region $0\le z_1\le S$ finally gives
that $g_{1_-2_-}$ is bounded.

The argument did not use that $\tau_1^2\le S$ and hence also applies to $g_{1_-2_+}$. Finally note that the contribution 
to $g_{1_+2_+}$ from $[0,S]$ is just a constant, whereas the one from $(S,T]$ has the same structure as used for $g_{1_-2_-}$.
Hence also $g_{1_+2_+}$ is bounded. 
\end{example}

\begin{remark}
\label{Rem:24.4a}\normalfont
The calculations show that $F_{1+}$ also is an atom if $b_+^0=b^1$. 
However, since the contribution rate $b_-^0$
is calculated via the equivalence principle 
given the benefits $b^0_+,b^1$ and the transition intensities, this
would be a very special case. The somewhat less special situation 
$b^0_+=b^1$ (same annuity to an disabled as to someone retired as active) 
also gives an atom, now at $F_{1+}$. In fact, $b_-^0=b^1$ is assumed
in BuchM.
\end{remark}

\begin{remark}
\label{Rem:24.4b}\normalfont
For a counterexample to \eqref{5.4a}, consider again the disability model with the only benefit being 
a lump sum of size $b^{01}(t)=$ $\e^{rt}\varphi(t)$ being paid out
at $\tau_0^1$ where $\varphi(t)=$ $(t-a)^21_{t\le a}+b$, cf.\ Fig.~\ref{varphiFig}.

\varphiFig
Then $U(s,T]=\varphi(\tau_0^1)$ has an atom at $b$ (corresponding to $\tau_0^1>a$)
and an absolutely continuous part on $(b,b+a^2]$. Letting $\psi:\,[b,b+a^2]\mapsto [0,a]$ 
be the inverse of $\varphi$,   $t=\psi(y)$
satisfies 
\[y=(t-a)^2+b\ \Rightarrow (y-b)^{1/2}=-(t-a)\ \Rightarrow \ t=\psi(y)=a-(y-b)^{1/2}\,,\]
where the minus sign after the first $\Rightarrow$ follows since $\psi$ is decreasing.
With $h$ the density of $\tau_0^1$, we get as in Lemma~\ref{Lemma:24.4a} that the conditional density $f$
of the absolutely continuous part is given by
\[f(y)\ =\ \frac{1}{\Prob(\tau_0^1\le a)}\frac{h(\psi(y))}{\bigl|\varphi'(\psi(y))\bigr|}\ =\ 
\frac{1}{\Prob(\tau_0^1\le a)}\frac{h(a-(y-b)^{1/2})}{(y-b)^{1/2}}
\]
But there are $c_1,c_2>0$ such that $f_0(y)\le c_1$ on $[b,b+a^2]$ and $h(t)\ge c_2$ on $[b,b+a^2]$, 
and so we get
\[\int\frac{f^2}{f_0}\ = \int_{b}^{b+a^2}\frac{f^2(y)}{f_0(y)}\,\dd y\ \ge\ \int_{b}^{b+a^2}\frac{c_2^2}{c_1(y-b)}\ =\ \infty\,,\]
meaning that \eqref{5.4a} does not hold.
\end{remark}

%\begin{remark}
%\label{Ex:23.4a}\normalfont
%Example~\ref{Ex:23.4a} TBC assumed feedforward. In the case of possible feed-back, TBC
%\end{remark}

%\subsection{}
\begin{example}[The Markovian arrival process]
Our second main example is the number $X=N(t)$ of events before time $t$ in the Markovian arrival process.
This is a discrete r.v., but for the general theory one just needs to replace Lebesgue measure $\dd x$ in
\eqref{5.4a}, \eqref{5.4a} etc.\ above by counting measure on $\{0,1,2,\ldots\}$.

For $X=N(t)$, a candidate for the reference distribution could be the Poisson distribution
with the same mean $\lambda$ as $N(t)$ in stationarity. The orthogonal polynomials are then the
Charlier-Poisson polynomials 
\begin{equation}\label{8.4a}
p_n(x)\ =\ \lambda^{n/2}(n!)^{-1/2}\sum_{k=0}^n(-1)^{n-k}{n\choose k}k!\lambda^{-k}{x\choose k}\ =\ 
\lambda^{n/2}(n!)^{1/2}L_n^{(x-n)}(\lambda)
\end{equation}
where $L_n$ is the $n$th Laguerre polynomial,
cf.\ Szeg\"o~\cite{Szego39} pp.\ 34--35 and Schmidt~\cite{Schmidt33}.\end{example}

\section{A numerical example}
By \eqref{eq:piecewise-constant}, we may (essentially) assume without loss of generality that the parameters are constant. Thus we shall consider a disability--unemployment model, defined in terms of a time--homogeneous Markov process $\{ Z(t) \}_{t\geq 0}$, with state space 
$E=\{ 1,2,3,4,5\}$, where state 1 corresponds to active (premium paying), 2 unemployed, 3 disabled, 4 re--employed and 5 death. Death rates from all states are assumed to be the same  and equal to $0.5$, and all other possible transitions happen at rate $0.1$  as illustrated in Figure \ref{examfig}. The interest rate is $r=0.08$, and the only lump sum payment (of quantity 2) is when entering the disability state from either state $1$, $3$ or $4$. In state $1$ premium is paid at rate $1$ and in state $3$ a benefit of rate $1$ is obtained. A parametrisation of the model is given by 
\[ 
\mat{C}=
\begin{pmatrix}
-0.7 & 0 & 0 & 0 & 0 \\
0 & -0.5 & 0 & 0 & 0 \\
0 & 0  & -0.7 & 0 & 0\\
0 & 0 & 0 & -0.6 & 0 \\
0 & 0 & 0 & 0 & 0 &
\end{pmatrix}, \ \ 
\mat{D} =
\begin{pmatrix}
0 & 0.1 & 0.1 & 0 & 0.5 \\
0 & 0 & 0 & 0 & 0.5 \\
0 & 0.1 & 0 & 0.1 & 0.5 \\
0 & 0.1 & 0 & 0 & 0.5 \\
0 & 0 & 0 & 0 & 0 
\end{pmatrix}, \ \
\mat{B}=
\begin{pmatrix}
0 &  2  & 0 & 0 & 0 \\
     0 & 0&  0&  0&  0\\
     0 & 2 & 0&  0& 0 \\
     0 & 2& 0& 0 & 0 \\
     0& 0& 0& 0& 0
\end{pmatrix}
  \]
  and 
  \[\vect{b} = (-1,0,1,0,0) . \]
It is clear from this example that parametrisations may be ambiguous. Indeed, for all $d_{ij}>0$ with a corresponding $b^{ij}=0$ we could equally have moved these values into the $\mat{C}$ matrix instead. 
 \begin{figure}[H]
   \centering 
\begin{tikzpicture}[scale=0.7] % E_2
\small
\filldraw[fill=blue!5!white,thick] (0,3) rectangle (1,4)    (3,3) rectangle (4,4) (6,3) rectangle (7,4) ;
\filldraw[fill=blue!5!white,thick] (6,0) rectangle (7,1) (6,6) rectangle (7,7) ;
\draw (0.5,3.5) node {$1$} (3.5,3.5) node {$2$}  (6.5,3.5) node {$4$} (6.5,0.5) node {$5$} (6.5,6.5) node {$3$};

\draw[thick,->] (0.5,4.3) -- (0.5,6.5) -- (5.7,6.5);   %02
\draw[thick,->] (0.5,2.7) -- (0.5,0.5) -- (5.7,0.5);   %04
\draw[thick,->] (7.3,6.5) -- (8.5,6.5) -- (8.5,0.5) -- (7.3,0.5);   %24
\draw[thick,->] (1.3,3.5) -- (2.7,3.5);   %01
\draw[thick,->] (5.7,3.5) -- (4.3,3.5);  %31
\draw[thick,->] (6.5,5.7) -- (6.5,4.3);  %23
\draw[thick,->] (6.5,2.7) -- (6.5,1.3);  %34
\draw[thick,->] (5.7,5.7) -- (4.3,4.3);   %21
\draw[thick,->] (4.3,2.7) -- (5.7,1.3);   %41
\draw (3.3,6.1) node {$\theta$} (3.3,0.8) node {$\mu$};
\draw (2,3.9) node {$\lambda$} (5,3.9) node {$\lambda$} 
 (8.8,3.5) node {$\mu$} ;
\draw (4.6,5.2) node {$\lambda$} (4.6,1.8) node {$\mu$};
\draw (6.8,5) node {$\eta$} (6.8,2) node {$\mu$};
\end{tikzpicture}
\caption{Disability-Unemployment Markov process with $\eta=\lambda=0.1$ and $\mu=0.5$}
   \label{examfig}
 \end{figure}
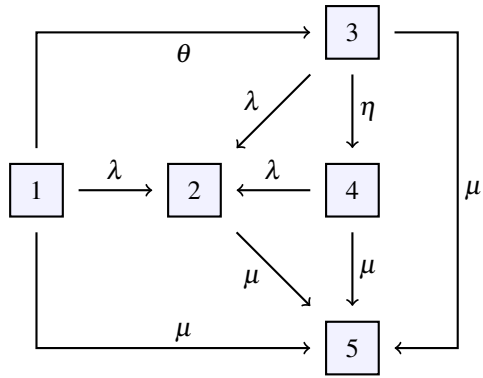

 \begin{figure}[h]
  \begin{minipage}[b]{0.4\textwidth}
   \hspace{-1cm}\includegraphics[scale=0.45]{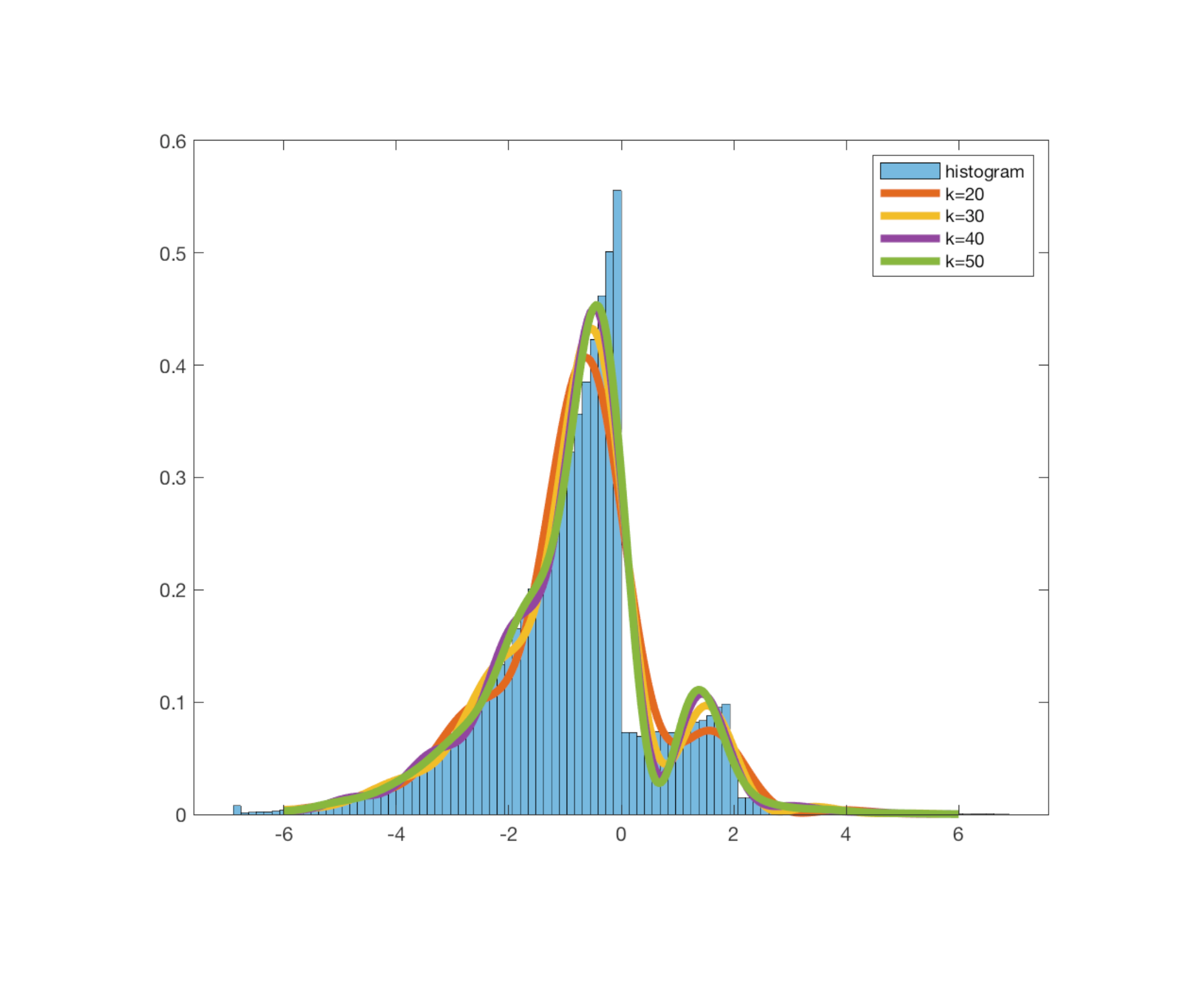}
    %\caption{Histogram of 300,000 simulations vs. densities obtained by $k$ moments}
    %\label{fig:Ex1-hist-vs-dens}
  \end{minipage}\hspace{1cm}
  \begin{minipage}[b]{0.4\textwidth}
    \includegraphics[scale=0.45]{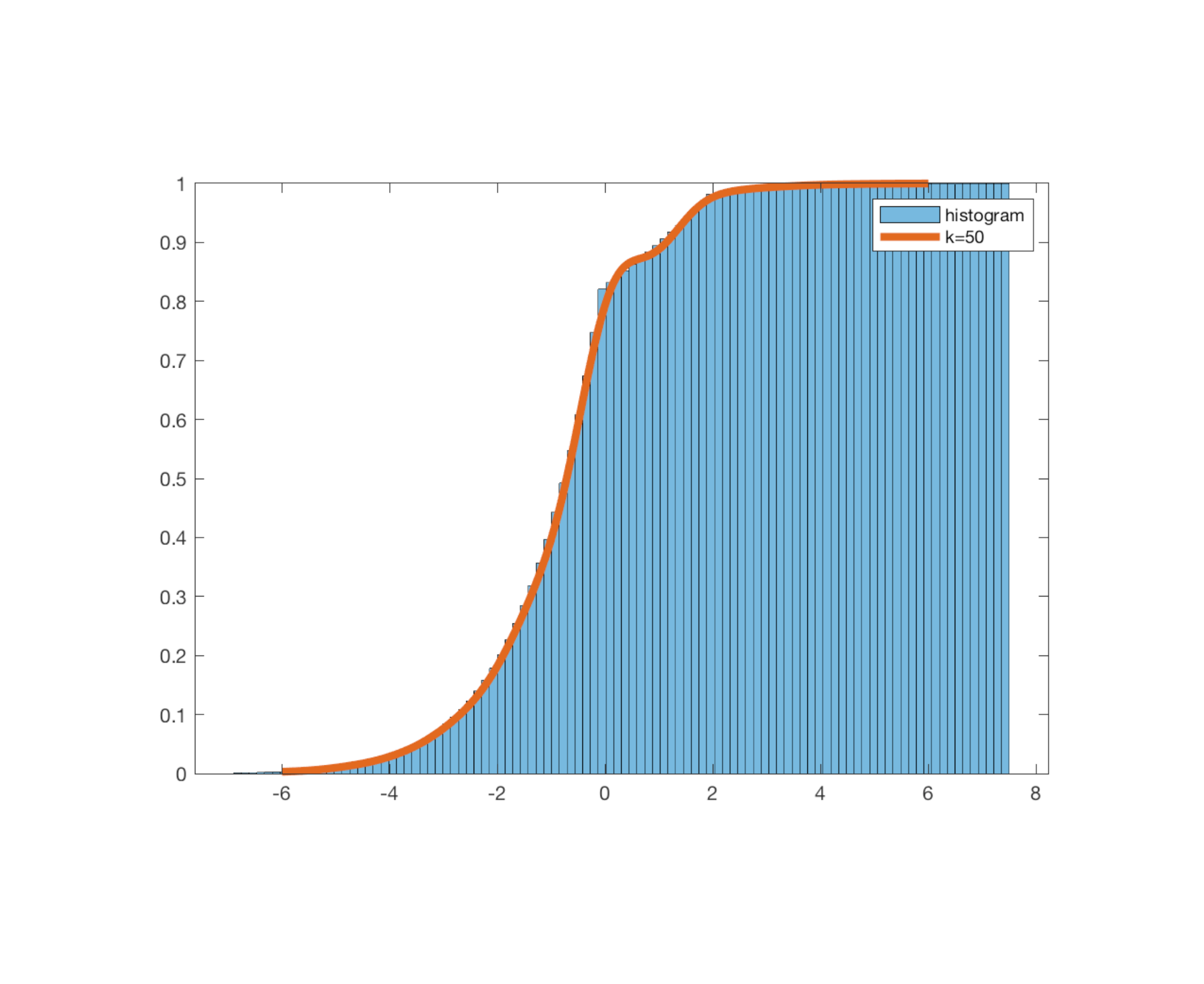}
    \end{minipage}
    \caption{Left: densities based on $k=20,30,40,50$ moments plotted towards a histogram based on 300,000 simulations. Right: Empirical CDF of the 300,000 simulations vs. the CDF  obtained by $k$ moments}
    \label{fig:Ex1-hist-vs-dens}
\end{figure}
A comparison of the first eight moments calculated respectively by a matrix exponential and by simulation is given in Table 
\ref{tab:moments}. 

\begin{table}[H]
  \begin{center}
    \begin{tabular}{c|r|r} % <-- Alignments: 1st column left, 2nd middle and 3rd right, with vertical lines in between
      \textbf{Order} & \textbf{Matrix exp} & \textbf{Simulation} \\ \hline
      1 & -0.8240 &  -0.8247\\
      2 & 2.8630 & 2.8639\\
      3 & -6.751 & -6.797\\
      4 &  33.21  &  33.33 \\
      5 &  -122.4 &  -124.5 \\
      6 &  708.9 & 716.6 \\
      7 &   -3233 &  -3323 \\
      8 &  20633 &  21024
    \end{tabular}
  \end{center}
  \caption{Moments calculated on exact method and on simulation}
    \label{tab:moments}
\end{table}
In Figure \ref{fig:Ex1-hist-vs-dens} we show a histogram of 300,000 simulated data as compared to estimated densities (and one cdf) based on a different number of moments. It is clear that the challenging shape of the density will in general require a high number of moments in order to obtain a good approximation. The cdf, which has a smoother behaviour and therefore is easier to approximate, seems to provide a good fit to the simulated data. How sensitive the cdf approximation is to the number of moments used is shown in Figure \ref{fig:cdfs}. The method seems at first eye to be pretty robust, which is also confirmed in Table \ref{tab:quantiles}, where quantiles which are obtained from the approximating cdfs for different numbers of moments, are all within a sensible range from each other.

\begin{table}[h!]
  \begin{center}
    \begin{tabular}{c|r|r} % <-- Alignments: 1st column left, 2nd middle and 3rd right, with vertical lines in between
      \textbf{Number of moments} & \textbf{$2.5\%$ quantile} & \textbf{$97.5\%$ quantile} \\ \hline
      10 & -4.05 &  2.15\\
      20 & -3.95 & 2.10\\
      30 & -4.00 & 2.05\\
      40 & -3.95 &  2.05 \\
      50 & -4.00 &  2.05  \\
    \end{tabular}
  \end{center}
  \caption{Quantiles calculated using the approximating cdf based on a different number of moments. The simulated quantiles are $-4.10$ and $1.95$ respectively.}
    \label{tab:quantiles}
\end{table}

\begin{figure}[h]
    \includegraphics[scale=0.50]{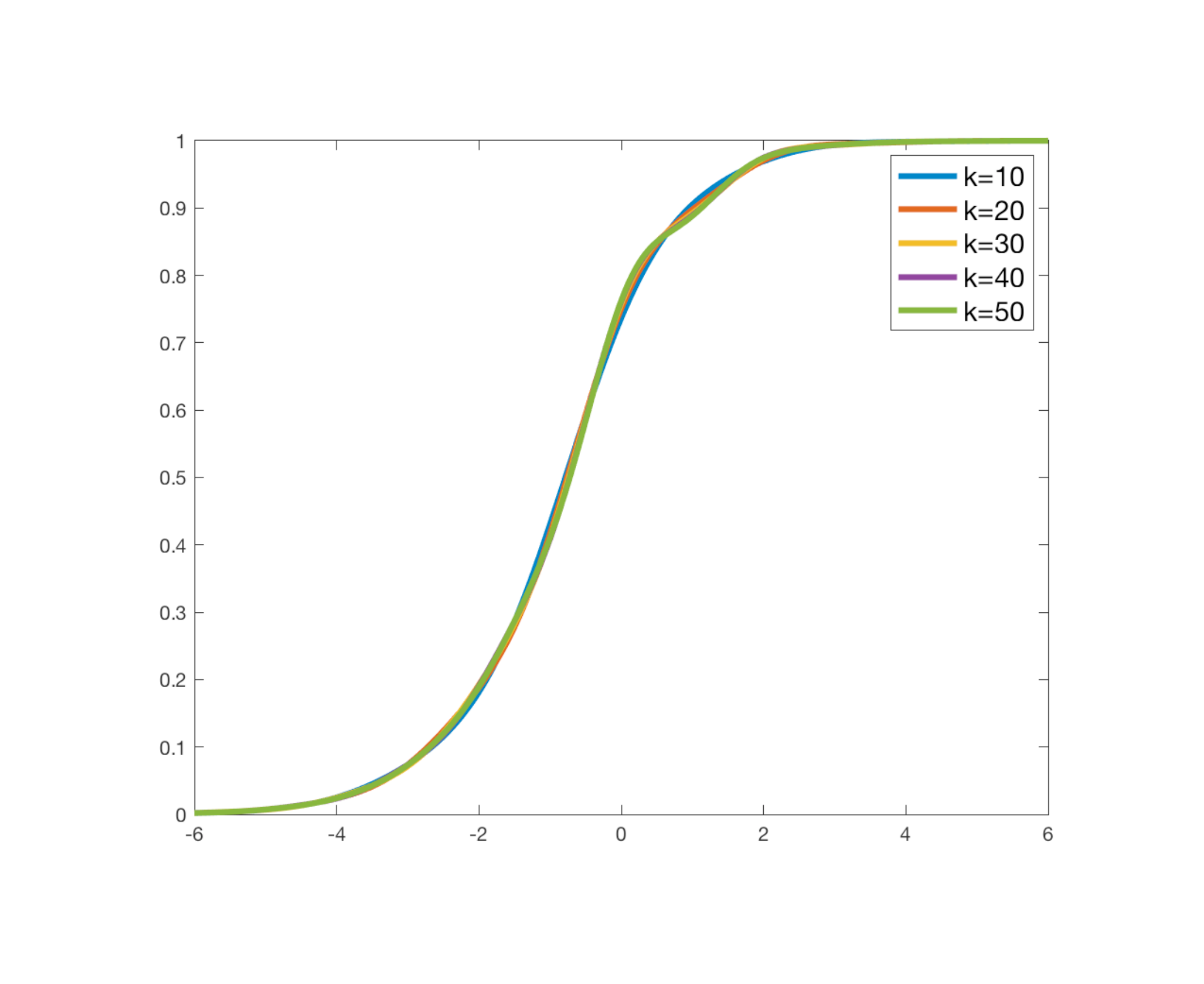}
    \caption{Approximation of the CDF based on a different number of moments.}
    \label{fig:cdfs}
\end{figure}

\section{Conclusion}
In this paper we have established a matrix oriented approach for calculating the (discounted) rewards of time--inhomogeneous Markov processes with finite state--space. In particular, for applications to multi--state Markov models in life insurance
our approach provides an alternative to standard derivations in the literature, which are usually based on case--by--case derivations involving differential or integral equations. In the slightly more general set--up in this paper we provide a unifying approach to deriving reserves and moments (of, in principle, arbitrary orders) which has a simple numerical implementation. The Laplace--Stieltjes transform of the (discounted) future payments, which plays an important role in the derivation of the moments and whose derivation is based on probabilistic (sample path) arguments, has a strikingly simple form which would allow for a numerical inversion in order to obtain the cdf or density of the future payments as well. However, since the moments of all orders, in principle, are available we propose an alternative method involving approximation of the cdf and densities via orthogonal polynomial expansions based on central moments. While this method seems to be very robust concerning the cdf, the approximation of the density itself is more involved which stems from the fact that presence of lump sums mixed with continuous rates implies that the densities can have a very challenging form.

\bibliographystyle{plainnat}
\bibliography{PHBib}

\end{document}